\documentclass[english,10pt,a4paper,leqno]{amsart}
\usepackage{amsmath,amstext, amsthm, amssymb}

\usepackage{color}
\definecolor{liens}{rgb}{1,0,0}
\usepackage[colorlinks=true, linkcolor=blue, 
hyperfootnotes=true,citecolor=blue,urlcolor=black]{hyperref}

\usepackage{latexsym, amsthm, amsfonts, amsmath,amssymb,graphicx,xcolor}
\input xy
\xyoption{all}
\usepackage[latin1]{inputenc} 
\usepackage[T1]{fontenc} 
\usepackage[english]{babel} 

\newtheorem{theo}{Theorem}[section]
\newtheorem*{theointro}{Theorem}
\newtheorem{lem}[theo]{Lemma}
\newtheorem{propo}[theo]{Proposition}
\newtheorem{coro}[theo]{Corollary}
\newtheorem*{corointro}{Corollary}
\theoremstyle{definition}

\newtheorem{hypothese}[theo]{Assumption}

\theoremstyle{remark}
\newtheorem{rem}[theo]{Remark}

\def\pd{(\phi,\delta)}

\def\de{\delta}

\def\Z{\mathbb{Z}}
\def\C{\mathbb{C}}
\def\N{\mathbb{N}}
\def\Q{\mathbb{Q}}

\def\n{\eta}

\def\s{\sigma}

\def\d{\delta}

\def\b{\beta}

\def\n'{\nu}
\def\d{\delta}

\def\k{\kappa}

\def\L{\Lambda}

\def\Gal {\mathrm{Gal}}
\def\GL {\mathrm{GL}}
\def\SL {\mathrm{SL}}

\def\cQ{\mathcal{Q}}
\def\Galdelta {\mathrm{Gal}^\delta}

\newcommand{\circKol}{\circ_{\delta}}
\newcommand{\derKol}{der_{\delta}}

\def\K{\mathbf{K}}
\def\k{\mathbf{k}}
\def\Const{\mathbf{C}}

\def\L{\mathbf{L}}

\begin{document}

\sloppy

\title{Hypertranscendence of solutions of Mahler equations}
\author{Thomas Dreyfus}
\address{Institut de Recherche Math\'ematique Avanc\'ee, U.M.R. 7501 Universit\'e de Strasbourg et C.N.R.S. 7, rue Ren\'e Descartes 67084 Strasbourg, FRANCE}
\email{dreyfus@math.unistra.fr}
\author{Charlotte Hardouin}
\address{Universit\'e Paul Sabatier - Institut de Math\'ematiques de Toulouse, 118 route de Narbonne, 31062 Toulouse.}
\email{hardouin@math.univ-toulouse.fr}
\author{Julien Roques}
\address{Universit\'e Grenoble Alpes, Institut Fourier,  CNRS UMR 5582, 100 rue des Maths, BP 74, 38402 St Martin d'H\`eres}
\email{Julien.Roques@ujf-grenoble.fr}

\thanks{The first (resp. the second) author would like to thank the ANR-11-LABX-0040-CIMI within
the program ANR-11-IDEX-0002-0 for its total (resp. partial) support. The second author's work is also supported
by ANR Iso-Galois. This project has received funding from the European Research Council (ERC) under the European Union's Horizon 2020 research and innovation programme under the Grant Agreement No 648132.}
\date{\today}
\maketitle

\vspace{-0.3cm}

\begin{abstract} 
The last years have seen a growing interest from mathematicians in Mahler functions. This class of functions includes the generating series of the automatic sequences. 
The present paper is concerned with the following problem, which is rather frequently encountered in combinatorics: a set of Mahler functions $u_{1},...,u_{n}$ being given, are $u_{1},...,u_{n}$ and their successive derivatives algebraically independent? 
In this paper, we give general criteria ensuring an affirmative answer to this question. We apply our main results to the generating series attached to the so-called Baum-Sweet and Rudin-Shapiro automatic sequences. In particular, we show that these series are hyperalgebraically independent, {\it i.e.}, that these series and their successive derivatives are algebraically independent. 
Our approach relies of the parametrized difference Galois theory (in this context, the algebro-differential relations between the solutions of a given Mahler equation are reflected by a linear differential algebraic group). 
\end{abstract} 

\tableofcontents

\section*{Introduction}

This paper grew out of an attempt to understand the algebraic relations between classical Mahler functions and their successive derivatives. By Mahler function, we mean a function $f(z)$ such that  
\begin{equation}\label{equa generique intro}
a_{n}(z) f(z^{p^{n}}) + a_{n-1}(z) f(z^{p^{n-1}}) + \cdots + a_{0}(z) f(z) = 0
\end{equation}
for some integers $p\geq 2$, $n \geq 1$, and some $a_{0}(z),\dots,a_{n}(z) \in \C(z)$ with ${a_{0}(z)a_{n}(z)\neq 0}$. 

The study of this class of 
functions was originally motivated by the work of Mahler in 
\cite{MahlerArith1929,MahlerArith1930,MahlerUber1930} about the algebraic relations between 
special values at algebraic points of Mahler functions. This arithmetic aspect of the theory of the Mahler functions was developed further by several authors, {\it e.g.}, Becker, Kubota, Loxton, van der Poorten, Masser, Nishioka, T\"ofer. We refer to Nishioka's book \cite{NishiokaLNM1631} and Pellarin's paper \cite{pellarinAIMM} for more informations and references. We shall simply mention that, quite recently, Philippon \cite{PhilipponGaloisNA} proved a refinement of Nishioka's analogue of the Siegel-Shidlovski theorem, in the spirit of Beukers' refinement of the Siegel-Shidlovski theorem~\cite{BRVSiegelShid}. See also \cite{adamczewski2016m,adamczewski2017methode}. Roughly speaking, it says that the algebraic relations over $\overline{\Q}$ between 
the above-mentioned special values
come from algebraic relations over $\overline{\Q}(z)$ between the functions themselves. These functional relations are at the heart of the present paper. 

The renewed attractiveness of the theory of Mahler functions comes (to a large extent) from its close connection with automata theory: the generating series ${f(z)=\sum_{k \geq 0} s_{k} z^{k}}$ of any 
$p$-automatic sequence $(s_{k})_{k \geq 0} \in \overline{\Q}^{\N}$ (and, actually, of any $p$-regular sequence) is a Mahler function; 
see Mend\`es France \cite{MF80}, Rand\'e \cite{RandeThese}, Dumas \cite{DumasThese}, Becker \cite{BeckerKReg}, Adamczewski and Bell \cite{BorisAboutMahler}, and the references therein. The famous examples are the generating series of the Thue-Morse, the paper-folding, the Baum-Sweet and the Rudin-Shapiro sequences (see Allouche and Shallit's book \cite{AS03}). 

The Mahler functions also appear in many other circumstances such as the combinatorics of partitions, the enumeration of words and
the analysis of algorithms of the type  divide and conquer; see for instance \cite{DF96} and the references therein. 

It is a classical problem (in combinatorics in particular) to determine whether or not a given generating series is transcendental or even  hypertranscendental over $\C(z)$\footnote{We say that a series $f(z) \in \C((z))$ is hypertranscendental over $\C(z)$ if $f(z)$ and all its derivatives are algebraically independent over $\C(z)$.}. 

The hypertranscendence over $\C(z)$ of Mahler functions solutions of  inhomogeneous Mahler equations of order one can be studied by using the work of Nishioka~\cite{NishiokaLNM1631}; see also the work of Nguyen~\cite{NG,NGT} {\it via} difference Galois theory. This can be applied to the paper-folding generating series for instance. Actually, Rand\'e already studied in \cite{RandeThese} the functions $f(z)$ meromorphic over the unit disc $D(0,1) \subset \C$ which are solutions of some inhomogeneous Mahler  equation of order one with coefficients in $\C(z)$: he proved that, if $f(z)$ is hyperalgebraic over $\C(z)$, then $f(z) \in \C(z)$ (see \cite[Chapitre 5, Th\'eor\`eme 5.2]{RandeThese}).   

The present work started with the observation that, besides this case, very few things are known. For instance, the hypertranscendence of  the Baum-Sweet or of the Rudin-Shapiro generating series was not known. 
The main objective of the present work is to develop an approach, as systematic as possible, in order to prove the hypertranscendence of such series.

To give an idea of the contents of this paper, we mention the following result (see Theorem \ref{user friend theo}), 
which is a consequence of one of our main hypertranscendence criteria. In what follows, we consider the field $\K=\cup_{j \geq 1} \C(z^{1/j})$ endowed with the field automorphism $\phi$ given by $\phi(f(z))=f(z^{p})$. We obtain in this way a difference field with field of constants $\K^{\phi}=\C$, and we have at our disposal a difference Galois theory over $\K$ (see Section \ref{rappels et complements galois differences}).

\begin{theointro}\label{theo intro un}
Assume that the difference Galois group over $\K$ of the Mahler equation (\ref{equa generique intro}) contains $\SL_{n}(\C)$ and that $a_{n}(z)/a_{0}(z)$ is a monomial. 
Let $f(z) \in \C((z))$ be a nonzero solution of (\ref{equa generique intro}). Then, the series $f(z), f(z^{p}),\ldots,f(z^{p^{n-1}})$ and all their successive derivatives are algebraically independent over $\C(z)$. In particular, $f(z)$ is hypertranscendental over $\C(z)$. 
\end{theointro}

The hypothesis that $a_{n}(z)/a_{0}(z)$ is a monomial is satisfied in any of the above-mentioned cases. 
Moreover, in the case $n=2$, there is an algorithm to determine whether or not the difference Galois group over $\K$ of equation (\ref{equa generique intro}) contains $\SL_2(\C)$; see \cite{Ro15}. 
It turns out that the difference Galois groups involved in the Baum-Sweet and in the Rudin-Shapiro cases both contain $\SL_{2}(\C)$ (see \cite[Section 9]{Ro15}). 
Therefore, we have the following consequences of the above theorem (see Theorems \ref{Galois BS} and \ref{Galois RS}). 
In what follows, we let $f_{BS}(z)$ and $f_{RS}(z)$ denote the generating series of the Baum-Sweet and of the Rudin-Shapiro sequences.

\begin{corointro}
The series $f_{BS}(z)$, $f_{BS}(z^{2})$ and all their successive derivatives are algebraically independent over $\C(z)$. In particular, $f_{BS}(z)$ is hypertranscendental over $\C(z)$. 
\end{corointro}

\begin{corointro}
The series $f_{RS}(z)$, $f_{RS}(-z)$ and all their successive derivatives are algebraically independent over $\C(z)$. In particular, $f_{RS}(z)$ is hypertranscendental over $\C(z)$. 
\end{corointro}

Actually, our methods also allow to study the relations between these series. We prove the following result (see Theorem \ref{BS plus RS}). 

\begin{corointro}
The series $f_{BS}(z),f_{BS}(z^{2}),f_{RS}(z),f_{RS}(-z)$ and all their successive derivatives are algebraically independent over $\C(z)$.
\end{corointro}

We shall now say a few words about the proofs of these results. Our approach relies on the parametrized difference Galois theory developed by Hardouin and Singer in \cite{HS}. Roughly speaking, to the difference equation (\ref{equa generique intro}), they attach a linear {\it differential} algebraic group over a differential closure $\widetilde{\C}$ of $\C$ -- called the parametrized difference Galois group -- which reflects the algebro-differential relations between the solutions of the equation. 
The above theorem is actually a consequence of the following purely Galois theoretic statement (see Section \ref{sec:hyperrank2} for more general results). 

\begin{theointro}
Assume that the difference Galois group over $\K$ of the Mahler equation (\ref{equa generique intro}) contains $\SL_{n}(\C)$ and that $a_{n}(z)/a_{0}(z)$ is a monomial. 
Then, the parametrized difference Galois group of equation (\ref{equa generique intro}) is caught between $\SL_{n}(\widetilde{\C})$ and $\C^{\times} \SL_{n}(\widetilde{\C})$.
\end{theointro}

Roughly speaking, the fact that the parametrized difference Galois group of equation (\ref{equa generique intro}) contains $\SL_{n}(\widetilde{\C})$ says that the algebro-differential relations between the elements of a basis $f_{1},\dots,f_{n}$ of solutions (in a suitable sense) of the equation (\ref{equa generique intro}) are generated by the relations satisfied by the determinant of the associated Wronskian matrix $(f_{j}(z^{p^{i-1}}))_{1 \leq i,j \leq n}$. In particular, there is no nontrivial algebro-differential relations between the entries of a given column of this matrix, and this is exactly the conclusion of the first theorem stated in this introduction (with $f_{1}=f$). 

Note that, in order to use the parametrized difference Galois theory developed by Hardouin and Singer, one cannot work with the base field $\K$ 
endowed with the automorphism $\phi$ and the usual derivation $d/dz$ because $\phi$ and $d/dz$ do not commute. In order to solve this problem, 
Michael Singer uses, in an unpublished proof\footnote{Letter from Michael Singer to the second author (February 25, 2010).} 
of the fact that the Mahler function $\sum_{n \geq 0} z^{p^{n}}$ is hypertranscendental, the field $\K(\log(z))$ and the derivation $z \log(z) d/dz$.
We follow this approach in the present paper.  
This idea also appears in Rand\'e's \cite{RandeThese}, but in a slightly different form. Indeed, Rand\'e uses the change of variable $z=\exp(t)$ in order to transform the Mahler difference operator $z \mapsto z^p$ into the  $p$-difference operator $t \mapsto pt$. Pulling back the usual Euler derivation $t d/dt$ to the $z$ variable, we find the derivation $z \log(z) d/dz$. 
Note that Lemma~\ref{lem:extconst} and Proposition \ref{propo:hypertransrank1} are also due to Michael Singer and appear in the above mentioned unpublished manuscript. \\    

This paper is organized as follows. Section \ref{sec1} contains reminders and complements on difference Galois theory. Section \ref{sec:parampv} starts with reminders and complements on parametrized difference Galois theory. Then, we state and prove user-friendly hypertranscendence criteria for general difference equations of order one. We finish this section with complements on (projective) isomonodromy for general difference equations from a Galoisian point of view. In Section \ref{sec hyptr}, we first study the hypertranscendence of the solutions of Mahler equations of order $1$. We then come to higher order equations and give our main hypertranscendence criteria for Mahler equations. 
Section \ref{sec applications} provides user-friendly hypertrancendence criteria and is mainly devoted to applications of our main results to the generating series of classical automatic sequences.  \\

 \noindent {\bf Acknowledgements.}  We would like to thank Michael Singer for discussions and support of this work. We also thank the anonymous referee 
 for interesting suggestions and references. 
\\

\noindent {\bf General conventions.} All rings are commutative with identity and contain the field of rational numbers. In particular, all fields are of characteristic zero.
\section{Mahler equations and difference Galois theory}\label{sec1}

\subsection{Difference Galois theory}\label{rappels et complements galois differences}

 For details on what follows, we refer to \cite[Chapter 1]{VdPS97}.

 A $\phi$-ring $(R,\phi)$ is a ring $R$ together with a ring automorphism $\phi : R \rightarrow R$. An ideal of $R$ stabilized by $\phi$ is called a $\phi$-ideal of $(R,\phi)$. If $R$ is a field, then $(R,\phi)$ is called a $\phi$-field. To simplify the notation, we will, most of the time, write $R$ instead of $(R,\phi)$.

The ring of constants of the $\phi$-ring $R$ is defined by 
$$R^{\phi}:=\{f \in R \ | \ \phi(f)=f\}.$$
If $R^{\phi}$ is a field, it is called the field of constants. 

A $\phi$-morphism (resp. $\phi$-isomorphism) from the $\phi$-ring $(R,\phi)$ to the $\widetilde{\phi}$-ring $(\widetilde{R},\widetilde{\phi})$ is a ring morphism (resp. ring isomorphism) $\varphi : R \rightarrow \widetilde{R}$ such that 
$\varphi \circ \phi = \widetilde{\phi} \circ \varphi$.

 Given a $\phi$-ring $(R,\phi)$, a $\widetilde{\phi}$-ring $(\widetilde{R},\widetilde{\phi})$ is a $R$-$\phi$-algebra  
if $\widetilde{R}$ is a ring extension of $R$ and $\widetilde{\phi}_{\vert R}=\phi$;
in this case, we will often denote $\widetilde{\phi}$ by $\phi$. Two $R$-$\phi$-algebras $(\widetilde{R}_{1},\widetilde{\phi}_{1})$ and $(\widetilde{R}_{2},\widetilde{\phi}_{2})$ are isomorphic  if there exists a $\phi$-isomorphism $\varphi$ from $(\widetilde{R}_{1},\widetilde{\phi}_{1})$ to $(\widetilde{R}_{2},\widetilde{\phi}_{2})$ such that 
$\varphi_{\vert R}=\operatorname{Id}_{R}$.\par

We fix a   $\phi$-field $\K$  such that $\k:=\K^\phi$ is algebraically closed.  We consider the following linear difference system 
\begin{equation}\label{eq1}
\phi (Y)=AY, \hbox{ with } A \in \GL_{n}(\K), n\in \N^{*}.
\end{equation}

By \cite[$\S$1.1]{VdPS97}, there exists a $\K$-$\phi$-algebra $R$ such that 
\begin{itemize}
\item[1)] there exists $U \in \GL_{n}(R)$ such that $\phi (U) = AU$ (such a $U$ is called a fundamental matrix of solutions of (\ref{eq1}));
\item[2)] $R$ is generated, as a $\K$-algebra, by the entries of $U$ and $\det(U)^{-1}$;
\item[3)] the only $\phi$-ideals of $R$ are $\{0\}$ and $R$.
\end{itemize}
Such a $R$ is called a Picard-Vessiot ring, or PV ring for short, for (\ref{eq1}) over $\K$. 
By  \cite[Lemma~1.8]{VdPS97}, we have $R^{\phi}=\k$. Two PV rings are isomorphic as $\K$-$\phi$-algebras.
A PV ring $R$ is not always an integral domain. However, there exist idempotents elements $e_{1},\dots,e_{s}$ of $R$ such that $R=R_{1} \oplus \cdots \oplus R_{s}$ where the $R_{i}:=Re_{i}$ are integral domains which are transitively permuted by $\phi$.
In particular, $R$ has no nilpotent element and one can consider 
its total ring of quotients $\cQ_R$, {\it i.e.}, the localization of $R$ with respect to the set of its nonzero divisors, which can be decomposed as the direct sum $\cQ_R=K_{1} \oplus \cdots \oplus K_{s}$ of the fields of fractions $K_{i}$ of the $R_{i}$. The ring $\cQ_R$ has a natural structure of $R$-$\phi$-algebra and we have ${\cQ_R}^{\phi}=\k$. Moreover, the $K_{i}$ are transitively permuted by $\phi$. We call the $\phi$-ring $\cQ_R$ a total PV ring for \eqref{eq1} over $\K$. 

The following lemma gives a characterization of the PV rings.

\begin{lem}[{\cite[Proposition 6.17]{HS}}] \label{lem: caracpvring}
 Let $S$ be a $\K$-$\phi$-algebra with no nilpotent element and let $\cQ_S$ be its total ring of quotients. 
If  the following properties hold:
\begin{enumerate}
\item there exists $V \in \GL_n(S)$ such that $\phi(V)V^{-1}=B \in \GL_n(\K)$ and such that $S$ is generated, as a $\K$-algebra, by the entries of $V$ and by $\det(V)^{-1}$,
\item ${\cQ_S}^ \phi=\k$,
\end{enumerate}
then $S$ is a PV ring for the difference system $\phi(Y)=BY$ over $\K$.
\end{lem}

As a corollary of the above lemma, we find
\begin{lem}\label{lem: carcacsubpvring}
Let $R$ be a PV ring over $\K$ and let $S$ be a $\K$-$\phi$-subalgebra of $R$. If there exists $V \in \GL_n(S)$ such that $\phi(V)V^{-1}=B \in \GL_n(\K)$ and such that $S$ is generated, as a $\K$-algebra, by the entries of $V$ and by $\det(V)^{-1}$ then $S$ is a PV ring for $\phi(Y)=BY$ over $\K$.
\end{lem}

\begin{proof}
Since $R$ has no nilpotent element, $S$ has no nilpotent element. By \cite[Corollary 6.9]{HS}, the total ring of quotients $\cQ_S$ of $S$ can be embedded into the total ring of quotients $\cQ_R$ of $R$.
Since ${\cQ_R}^\phi=\k$, we have ${\cQ_S}^\phi=\k$. Lemma \ref{lem: caracpvring} yields the desired result.
\end{proof}

The  difference Galois group $\Gal(\cQ_R/\K)$ of $R$ over $\K$ is the group of 
$\K$-$\phi$-automorphisms of $\cQ_R$ commuting 
with $\phi$:
$$
\Gal(\cQ_R/\K) :=\{ \sigma \in \mathrm{Aut}(\cQ_R/\K) \ | \ \phi\circ\sigma=\sigma\circ \phi \}.
$$
Abusing notation, we shall sometimes let $\Gal(\cQ_R/F)$ denote the group ${\{\sigma \in \mathrm{Aut}(\cQ_R/F) \ | \ \phi\circ\sigma=\sigma\circ \phi \}} $  for $F$ a $\K$-$\phi$-subalgebra of $\cQ_R$.

An easy computation shows that, for any $\sigma \in \Gal(\cQ_R/\K) $, there exists a unique ${C(\sigma) \in \GL_{n}(\k)}$ such that $\sigma(U)=UC(\sigma)$.
By  \cite[Theorem~1.13]{VdPS97}, the faithful representation
\begin{eqnarray*}
 \Gal(\cQ_R/\K) & \rightarrow & \GL_{n}(\k) \\ 
 \sigma & \mapsto & C(\sigma)
\end{eqnarray*}
identifies  $\Gal(\cQ_R/\K) $ with a linear algebraic subgroup of $\GL_{n}(\k)$.
If we choose another fundamental matrix of solutions $U$, we find a conjugate representation. 

A fundamental theorem
of difference Galois theory (\cite[Theorem 1.13]{VdPS97}) says that $R$ is the coordinate ring of a $G$-torsor over $\K$. In particular, the dimension of $ \Gal(\cQ_R/\K)$ as a linear algebraic group
over $\k$ coincides with the transcendence degree of the $K_{i}$ over $\K$. Thereby, the difference Galois group controls the algebraic
relations satisfied by the solutions.

The following proposition gives a characterization of the normal algebraic subgroups of $\Gal(\cQ_R/\K)$.

\begin{propo}\label{propo:normalsubgroup}
An algebraic  subgroup $H$ of $ \Gal(\cQ_R/\K)$ is normal if and only if the $\phi$-ring ${{\cQ_R}^{H}:=\{g\in \cQ_R \ |Ê\ \forall \s \in H, \s(g)=g\}}$ is stable under the action of $\Gal(\cQ_R/\K)$. 
In this case, the $\K$-$\phi$-algebra ${\cQ_R}^{H}$ is a total PV ring over $\K$  and the following sequence of group morphisms is exact
$$
\xymatrix{
 0   \ar[r] &  H  \ar[r]^-{\iota} &  \Gal(\cQ_R/\K)   \ar[r]^-\pi &  \Gal({\cQ_R}^{H}/\K)  \ar[r] & 0,
}
$$
where $\iota$ is the inclusion of  $H$  in $ \Gal(\cQ_R/\K) $ and $\pi$ denotes the restriction
of the elements of $\Gal(\cQ_R/\K)$ to   ${\cQ_R}^{H}$. 
\end{propo}

\begin{proof}
Assume that $H$ is normal in $\Gal(\cQ_R/\K)$. For all  $\tau \in \Gal(\cQ_R/\K)$, $g \in {\cQ_R}^{H}$, and $\sigma \in H$, we have
$$
\sigma(\tau(g))=\tau (( \tau^{-1} \sigma \tau)(g))=\tau(g).$$
This shows that ${\cQ_R}^{H}$ is stable under the action of $\Gal(\cQ_R/\K)$.
Conversely, assume that ${\cQ_R}^{H}$ is stable under the action of $\Gal(\cQ_R/\K)$. 
 Then, we can consider the restriction morphism 
\begin{eqnarray*}
\pi: \Gal(\cQ_R/\K) &\rightarrow &\Gal({\cQ_R}^{H}/\K) \\ 
\sigma &\mapsto &\sigma|_{{\cQ_R}^{H}}.
\end{eqnarray*} 
By Galois correspondence (see \cite[Theorem 6.20]{HS}), we have $\ker(\pi)=H$ and, hence, $H$ is normal in $\Gal(\cQ_R/\K)$. The rest of the proof is \cite[Corollary 1.30]{VdPS97}.
 \end{proof}

\begin{coro}\label{coro1}
 Let $f$ be an invertible element of $R$ such that $\phi (f)=af$ for some $a\in \K$. Let $\cQ_f \subset \cQ_R$ be the
 total ring of quotients of $\K[f,f^{-1}]$; this is a total PV ring for $\phi (y)=ay$ over $\K$. 
   Then, $\Gal(\cQ_R/\cQ_f)$ is a solvable algebraic group if and only if $\Gal(\cQ_R/\K) $ is a solvable algebraic group.
\end{coro}

\begin{proof}
We have ${\cQ_R}^{\Gal(\cQ_R/\cQ_f)}=\cQ_f$ (in virtue of the Galois correspondence \cite[Theorem 1.29.3]{VdPS97}) and $\cQ_f$ is stable under the action of $\Gal(\cQ_R/\K)$ (because, for all $\sigma \in \Gal(\cQ_R/\K)$, we have $\sigma(f)f^{-1} \in {\cQ_R}^\phi=\k$). 
By Proposition \ref{propo:normalsubgroup}, $\Gal(\cQ_R/\cQ_f)$ is normal in $\Gal(\cQ_R/\K)$ and the sequence 
$$
\xymatrix{
 0   \ar[r] & \Gal(\cQ_R/\cQ_f)  \ar[r]^\iota &  \Gal(\cQ_R/\K)   \ar[r]^\pi &  \Gal(\cQ_f/\K)  \ar[r] & 0
}
$$
is exact. Since $\Gal(\cQ_f/\K) \subset \GL_1(\k)$ is abelian, the group $\Gal(\cQ_R/\K) $ is solvable
if and only if the same holds for $\Gal(\cQ_R/\cQ_f)$.
\end{proof}

\subsection{More specific results about Mahler equations}\label{mahler as diff eq}

Now, we restrict ourselves to the Mahlerian context. 

We let $p\geq 2$ be an integer. 

We consider the field 
$$
\K:=\cup_{j \geq 1} \C\left(z^{1/j}\right).
$$  
The field automorphism 
\begin{eqnarray*} 
\phi : \K & \rightarrow & \K \\ 
f(z) & \mapsto & f(z^{p})
\end{eqnarray*} 
gives a structure of $\phi$-field on $\K$ such that $\K^\phi=\C$.

We also consider the field $\K':=\K(\log(z))$.  The field automorphism 
\begin{eqnarray*} 
\phi : \K' & \rightarrow & \K' \\ 
f(z,\log(z)) & \mapsto & f(z^{p},p \log(z))
\end{eqnarray*} 
gives a structure of $\phi$-field on $\K'$ such that $\K'^\phi=\C$.

In the sequel, we shall consider Mahler equations above the $\phi$-field $\K$ and also above its $\phi$-field extension $\K'$. We shall now study the effect of the base extension from $\K$ to $\K'$ on the difference Galois groups.

We first state and prove a lemma.

\begin{lem}\label{lem log(z)}
Let $L$ be a $\phi$-subfield of $\K'$ that contains $\K$. Then, there exists an integer $k \geq 0$ such that $L=\K(\log(z)^{k})$.
\end{lem}

\begin{proof}
The case $L=\K$ is obvious (take $k=0$). We shall now assume that $L\neq \K$.  Lemma \ref{lem: caracpvring} ensures that $\K'$ is a total PV ring  over $\K$ for the equation
$
{\phi(y)=py}
$.  The action of $\Gal(\K'/\K)$ on $\log(z)$ allows to see  $\Gal(\K'/\K)$ as an algebraic subgroup of  $\C^{\times}$. Since  $\log(z)$ is transcendental over $\K$, we have $\Gal(\K'/\K)=\C^{\times}$.
Since $L \neq \K$, the group $\Gal(\K'/L)$ is a proper algebraic subgroup of $\C^{\times}$ and, hence, is a group of roots of unity.
Then there exists an integer $k \geq 1$ such that $\Gal(\K'/L) = \mu_k :=\{ c \in \C^{\times} \ |Ê\ c^k =1\}$.  So, $\log(z)^k$
is fixed by $\Gal(\K'/L)$ and, hence, belongs to $L$ by Galois correspondence. Since $\Gal(\K'/\K(\log(z)^k)) \subset \mu_k$, we get that 
$L=\K( \log(z)^k)$.
\end{proof}

We consider the difference system 
\begin{equation}\label{l equa}
\phi (Y)
=
A
Y
\end{equation}
with $A \in \GL_{n}(\C(z))$. 
Let $R'$ be a PV ring for (\ref{l equa}) over 
$\K'$;
then $\mathcal{Q}_{R'}$ is a total PV ring for (\ref{l equa}) over $\K'$. Let $U \in \GL_{n}(R')$ be a fundamental matrix of solutions of (\ref{l equa}). Let $R$ be the $\K$-subalgebra of $R'$ generated by the entries of $U$ and $\det(U)^{-1}$. By \cite[Corollary 6.9]{HS}, we have $\cQ_R \subset \cQ_{R'}$. Since ${\cQ_{R'}}^\phi= \K'^\phi =\C$, we have ${\cQ_R}^\phi=\C$ and Lemma \ref{lem: caracpvring} allows to conclude that $R$ is a PV  ring for (\ref{l equa}) over $\K$ and $\mathcal{Q}_{R}$ is a total PV ring for (\ref{l equa}) over $\K$. 

The restriction morphism 
$$
\iota : \Gal(\cQ_{R'}/\K') \rightarrow \Gal(\cQ_{R}/\K)
$$ 
is a closed immersion; we will freely identify $\Gal(\cQ_{R'}/\K')$ with the subgroup $\iota(\Gal(\cQ_{R'}/\K'))$ of $\Gal(\cQ_{R}/\K)$.

\begin{propo}
The difference Galois group $\Gal(\cQ_{R'}/\K')$ is a normal subgroup of $\Gal(\cQ_{R}/\K)$ and the quotient $\Gal(\cQ_{R}/\K)/\Gal(\cQ_{R'}/\K')$ is either trivial or isomorphic to $\C^{\times}$.
\end{propo}

\begin{proof}
We set $G':=\iota(\Gal(\cQ_{R'}/\K'))$ and $G:=\Gal(\cQ_{R}/\K)$. 
Let us consider ${F:=(\mathcal{Q}_{R})^{G'}=(\mathcal{Q}_{R'})^{G'}\cap \mathcal{Q}_{R}=\K' \cap \mathcal{Q}_{R}}$. The Galois correspondence \cite[Theorem 6.20]{HS} ensures that  $G'=\Gal(\cQ_{R}/F)$. Since $F/\K$ is a $\phi$-subfield extension of $\K'/\K$, Lemma \ref{lem log(z)} ensures that there exists an integer $k \geq 0$ such that $F=\K(\log(z)^{k})$. Since $F^\phi=\C$, Lemma \ref{lem: caracpvring} shows that $F$ is a total PV ring over $\K$ for $\phi (y) = p^{k} y$. Using Proposition \ref{propo:normalsubgroup}, we see that $G'$ is a normal subgroup of $G$ and that $G/G'$ is isomorphic to the difference Galois group over $\K$ of $\phi (y) = p^{k} y$, which is  trivial if $k=0$  and equal to $\C^{\times}$ otherwise. 
\end{proof}

\begin{coro}\label{coro contient SL}
If $\SL_{n}(\C) \subset \Gal(\cQ_{R}/\K)$ then $\SL_{n}(\C) \subset \Gal(\cQ_{R'}/\K')$.
\end{coro}

\section{Parametrized difference Galois theory}\label{sec:parampv}
We will use standard notions and notation of difference and differential algebra which can be found in \cite{Cohn:difference} and \cite{VdPS97}. 

\subsection{Differential algebra}\label{section diff alg}

A $\delta$-ring $(R,\delta)$ is a ring $R$  endowed with a derivation $\delta:R\rightarrow R$ (this means that $\delta$ is additive and satisfies the Leibniz
rule ${\delta(ab)=\delta(a)b+a\delta(b)}$, for all $a,b \in R$).  If $R$ is a field, then $(R,\delta)$ is called a $\delta$-field. To simplify the notation, we will, most of the time,
write $R$ instead of $(R, \delta)$.

We let $R^\delta$ denote the ring of $\delta$-constants of the $\delta$-ring $R$, {\it i.e.}, 
$$
R^\delta:=\{c \in R \ | \ \delta(c)=0\}.
$$
If $R^{\delta}$ is a field, it is called the field of $\delta$-constants. 

Given a $\delta$-ring $(R,\delta)$, a $\widetilde{\delta}$-ring $(\widetilde{R},\widetilde{\delta})$ is a $R$-$\delta$-algebra  
if $\widetilde{R}$ is a ring extension of $R$ and $\widetilde{\delta}_{\vert R}=\delta$;
in this case, we will often denote $\widetilde{\delta}$ by $\delta$. 
Let $\K$ be a $\delta$-field. If $\L$ is a $\K$-$\delta$-algebra and a field, we say that 
$\L/\K$ is a $\delta$-field extension.  Let  $R$ be a $\K$-$\delta$-algebra and let 
 $a_1,\dots,a_n \in R$. We let $\K\{a_1,\dots,a_n\}$ denote the smallest $\K$-$\delta$-subalgebra  of $R$ containing  $a_1,\dots,a_n$. 
Let $\L/\K$ be a $\delta$-field extension and let $a_1,\dots,a_n \in \L$. We let $\K\langle a_1,\dots,a_n\rangle$ denote the smallest $\K$-$\delta$-subfield of $\L$ containing  $a_1,\dots,a_n$.

The ring of $\delta$-polynomials in the differential indeterminates $y_1,\ldots,y_n$ and with coefficients in a differential field $(\K,\delta)$,  denoted by $\K\{y_1,\ldots,y_n\}$, is the ring of polynomials in the indeterminates $\{\delta^j y_i\:|\: j \in \N, 1\le i\le n\}$ with coefficients in $\K$.

Let $R$ be be a $\K$-$\delta$-algebra and let $a_1, \dots,a_n \in R$. If there exists a nonzero $\delta$-polynomial $P \in \K\{y_1,\ldots,y_n\}$ such that 
$P(a_1,\dots,a_n)=0$, then we say that  $a_1,\dots,a_n$ are 
hyperalgebraically dependent over $\K$. Otherwise, we say that $a_1,\dots,a_n$ are hyperalgebraically independent over $\K$. 

 A $\delta$-field $\k$ is called differentially closed if, for every (finite) set of $\delta$-polynomials $\mathcal F$, if the system of differential equations $\mathcal F=0$ has a solution with entries in some $\delta$-field extension $\L$, then it has a solution with entries in $\k$. Note that the field of $\delta$-constants $\k^{\delta}$ of any differentially closed $\delta$-field $\k$ is algebraically closed. Any $\delta$-field $\k$ has a differential closure
$\widetilde{\k}$, {\it i.e.}, a differentially closed  $\delta$-field extension, and we have $\widetilde{\k}^{\delta}=\k$.

From now on, we consider a differentially closed $\delta$-field $\k$. 

A subset $W \subset \k^{n}$ is Kolchin-closed (or $\delta$-closed, for short) if there exists ${S \subset \k\{y_{1},\dots,y_{n}\}}$ such that
$$
W=
\left\{ a \in \k^n\:|\: \forall f \in S, f(a)=0\right\}.
$$
The Kochin-closed subsets of $\k^{n}$ are the closed sets of a topology on $\k^{n}$, called the Kolchin topology. 
The Kolchin-closure of $W \subset \k^{n}$ is the closure of $W$ in $\k^{n}$ for the Kolchin topology.

Following Cassidy in \cite[Chapter~II, Section~1, p.~905]{C72}\label{def:LDAG}, we say that a subgroup $G \subset \GL_{n}(\k) \subset \k^{n \times n}$ is a linear differential algebraic group ({\it LDAG}) if $G$ is the intersection
of a Kolchin-closed subset of $\k^{n \times n}$ (identified with $\k^{n^{2}}$) with $\GL_n(\k)$.

A $\delta$-closed subgroup, or $\delta$-subgroup for short, of an LDAG is a subgroup that is Kolchin-closed. The Zariski-closure of a LDAG $G\subset\GL_n(\k)$ is denoted by $\overline{G}$ and is a linear algebraic group.

We will use the following fundamental result. 

\begin{propo}[{\cite[Proposition 42]{C72}}]\label{propo3}
Let $\k$ be a differentially closed field. Let $\Const:=\k^\delta$. A  Zariski-dense $\delta$-closed subgroup of  $\SL_{n}(\k)$ is either conjugate to  $\SL_{n}(\Const)$ or  equal to $\SL_{n}(\k)$.
\end{propo}

We will also use the following result.

\begin{lem}[{\cite[Lemma 11]{MitSing}}]\label{lem normalisateur}
Let $\k$ be a differentially closed field. Let $\Const:=\k^\delta$. Then, the normalizer of $\SL_{n}(\Const)$ in $\GL_{n}(\k)$ is $\k^{\times} \SL_{n}(\Const)$. 
\end{lem}

\subsection{Difference-differential algebra}

 A $(\phi, \d)$-ring $(R,\phi,\d)$ is a ring $R$ endowed with a ring automorphism $\phi$ and a derivation $\delta : R \rightarrow R$ (in other words, $(R,\phi)$ is a $\phi$-ring and $(R,\delta)$ is a $\delta$-ring) such that $\phi$ commutes with $\delta$. 
If $R$ is a field, then $(R,\phi,\d)$ is called a $(\phi, \d)$-field. If there is no possible confusion, we will write  $R$ instead of $(R,\phi, \d)$.

We have straightforward notions of $\pd$-ideals, $\pd$-morphisms, $\pd$-algebras, {\it etc}, similar to the notions recalled in Sections \ref{sec1} and \ref{section diff alg}. We omit the details and refer for instance to \cite[Section 6.2]{HS}, and to the references therein, for details.

In order to use the parametrized difference Galois theory developed in \cite{HS}, we will need to work with a base $(\phi, \d)$-field $\K$ such that $\k:=\K^\phi$ is differentially closed. 
Most of the common
function fields 
do not satisfy this condition. The following result shows that  any $\pd$-field with algebraically closed field of constants  can be embedded into a $\pd$-field
with differentially closed field of constants.  
The following lemma appears in an unpublished proof due to M. Singer (letter from Michael Singer to the second author, February 25, 2010) of the fact that the Mahler function $\sum_{n \geq 0} z^{p^{n}}$ is hypertranscendental.
It is close to \cite[Proposition 2.4]{CHS} and \cite[Lemma 1.11]{VdPS97}, but it is not completely similar.

 \begin{lem}\label{lem:extconst}
  Let $F$ be a $\pd$-field with $\k:=F^\phi$ algebraically closed. Let $\widetilde{\k}$ be a differentially closed field 
  containing $\k$. Then, the ring $\widetilde{\k} \otimes_\k F$ is an integral domain whose
  fraction field $\K$ is a $\pd$-field extension of $F$ such that $\K^\phi=\widetilde{\k}$.
  \end{lem}

\begin{proof}
The first assertion follows from the fact that, since $\k$ is algebraically closed, the extension $\widetilde{\k}/\k$ is regular. 

In what follows, we see $F$ in $\widetilde{\k} \otimes_\k F$ and $\widetilde{\k}$ in $\widetilde{\k} \otimes_\k F$ {\it via} the maps 
$$
\begin{array}{rcl}
F & \rightarrow & \widetilde{\k} \otimes_\k F \\ 
f & \mapsto & 1 \otimes f
\end{array} 
\text{ \ and \ }
\begin{array}{rcl}
\widetilde{\k} & \rightarrow & \widetilde{\k} \otimes_\k F \\ 
a & \mapsto & a \otimes 1.
\end{array}
$$

The maps  
$$
\begin{array}{rcl}
\phi: \widetilde{\k} \otimes_\k F & \rightarrow & \widetilde{\k} \otimes_\k F \\ (a,b)& \mapsto& a \otimes \phi(b)
\end{array}
\text{ \ and \ } 
\begin{array}{rcl}
\delta: \widetilde{\k} \otimes_\k F & \rightarrow & \widetilde{\k} \otimes_\k F \\  
(a,b) &\mapsto& \delta(a) \otimes b + a \otimes \delta(b)
\end{array}
$$
are well-defined and endow $\widetilde{\k} \otimes_\k F$ with a structure of $F$-$\pd$-algebra.

To prove the second statement, we first show that any $\phi$-ideal of $\widetilde{\k} \otimes_\k F$ 
is trivial. Let $(c_i)_{i \in I}$ be a $\k$-basis of $\widetilde{\k}$.  
Let $\mathfrak I$ be a nonzero  $\phi$-ideal of  $\widetilde{\k} \otimes_\k F$ and let $w=\sum_{i=1}^n c_i \otimes f_i$ be a nonzero element of 
$\mathfrak I$  with $f_i \in F$ and $n$ minimal. Without loss of generality, we can assume that $f_1=1$. Since $\phi(w)-w=\sum_{i=2}^n c_i \otimes(\phi(f_i)-f_i)$ is an element of $\mathfrak I$ with   
fewer terms than $w$, it must be equal to $0$. This implies that, for all $i \in \{1, \dots, n\}$, $\phi(f_i)=f_i$, {\it i.e.}, $f_i \in \k$. Then, $w=(\sum_{i=1}^n c_i f_i)\otimes 1$ is invertible in $\widetilde{\k} \otimes_\k F$ and, hence, $\mathfrak{I}=\widetilde{\k} \otimes_\k F$.

Let $c \in \K^\phi$. Since $\mathfrak{I}:=\{ d \in \widetilde{\k} \otimes_\k F \ | \ dc \in \widetilde{\k} \otimes_\k F\}$ is a nonzero $\phi$-ideal of $\widetilde{\k} \otimes_\k F$, we must have $\mathfrak{I}=\widetilde{\k} \otimes_\k F$. In particular, $1 \in \mathfrak{I}$ and, hence, $c \in \widetilde{\k} \otimes_\k F$. Writing 
$c=\sum_{i \in I} c_i \otimes f_i$, we see that $\phi(c)=c$ implies $\phi(f_i)=f_i$ for all $i \in \{1,\dots, n\}$. Therefore, the  $f_i$
are in $\k$ and, hence, $c$ belongs to $\widetilde{\k}$.
\end{proof}

\subsection{Parametrized difference Galois theory}

For details on what follows, we refer to  \cite{HS}.

Let $\K$ be a $(\phi,\d)$-field with $\k:=\K^{\phi}$ differentially closed.
 We consider the following linear difference system 
\begin{equation}\label{eq9}
\phi (Y)=AY
\end{equation} 
with $A \in \GL_{n}(\K)$ for some integer $n \geq 1$.

By  \cite[$\S$ 6.2.1]{HS}, there exists a $\K$-$(\phi,\d)$-algebra $S$  such that 
\begin{itemize}
\item[1)] there exists $U \in \GL_{n}(S)$ such that $\phi (U)=AU$ (such a $U$ is called a
fundamental matrix of solutions of (\ref{eq9}));
\item[2)] $S$ is generated, as $\K$-$\de$-algebra, by the entries of $U$ and $\det(U)^{-1}$;
\item[3)] the only $(\phi,\d)$-ideals of $S$ are $\{0\}$ and $S$.
\end{itemize}
Such a $S$ is called a parametrized Picard-Vessiot ring, or PPV ring for short, for (\ref{eq9})  over $\K$. It is unique up to isomorphism of $\K$-$(\phi,\d)$-algebras.  A PPV ring is not always an integral domain. However, there exist idempotent elements $e_{1},\dots,e_{s}$ of $R$ such that $R=R_{1} \oplus \cdots \oplus R_{s}$ where the $R_{i}:=Re_{i}$ are integral domains stable by $\delta$ and transitively permuted by $\phi$.
In particular, $S$ has no nilpotent element and one can consider 
its total ring of quotients $\cQ_S$. It can be decomposed as the direct sum $\cQ_S=K_{1} \oplus \cdots \oplus K_{s}$ of the fields of fractions $K_{i}$ of the $R_{i}$. The ring $\cQ_S$ has a natural structure of $R$-$(\phi,\delta)$-algebra and we have ${\cQ_S}^{\phi}=\k$. Moreover, the $K_{i}$ are transitively permuted by $\phi$. We call the $(\phi,\delta)$-ring $\cQ_S$ a total PPV ring for \eqref{eq9} over $\K$. 

The parametrized difference Galois group $\Gal^{\d}(\cQ_S/\K)$ of $S$ over $(\K,\phi,\d)$ 
is the group of 
$\K$-$\pd$-automorphisms of $\cQ_{S}$:
$$
\Gal^{\d}(\cQ_S/\K):=\{ \sigma \in \mathrm{Aut}(\cQ_S/\K) \ | \ \phi\circ\sigma=\sigma\circ \phi \text{ and } \d\circ\sigma=\sigma\circ \d \}.
$$

Note that, if $\delta=0$, then we recover the difference Galois groups considered in Section \ref{rappels et complements galois differences}.

A straightforward computation shows that, for any $\sigma \in \Gal^{\d}(\cQ_S/\K)$, there exists a unique $C(\sigma) \in \GL_{n}(\k)$ such that $\sigma(U)=UC(\sigma)$.  By \cite[Proposition~6.18]{HS}, the faithful representation
\begin{eqnarray*}
 \Gal^\de(\cQ_S/\K) & \rightarrow & \GL_{n}(\k) \\ 
 \sigma & \mapsto & C(\sigma)
\end{eqnarray*}
identifies  $\Gal^\de(\cQ_S/\K) $ with a linear differential algebraic subgroup of $\GL_{n}(\k)$. If we choose another fundamental matrix of solutions $U$, we find a conjugate representation. 

The parametrized difference Galois group $\Gal^\de(\cQ_S/\K)$ of \eqref{eq9} reflects the differential algebraic relations between the solutions of \eqref{eq9}. In particular, the $\delta$-dimension of $\Gal^\de(\cQ_S/\K)$ coincides with the $\delta$-transcendence degree of the $K_{i}$ over $\K$ (see \cite[Proposition 6.26]{HS} for definitions and details). 

A parametrized Galois correspondence holds between the $\delta$-closed subgroups of  $\Gal^\de(\cQ_S/\K) $
and the  $\K$-$\pd$-subalgebras $F$ of  $\cQ_S$ such that every nonzero divisor of $F$ is a unit of $F$ (see for instance \cite[Theorem~6.20]{HS}).  Abusing notation, we still let 
$\Galdelta(\cQ_S/F)$ denote the group of $F$-$\pd$-automorphisms of $\cQ_S$. The following proposition
is at the heart of the parametrized Galois correspondence.

\begin{propo}[{\cite[Theorem 6.20]{HS}}]\label{propo:ppvcorres}
Let $S$ be a PPV ring over $\K$.  Let   $F$  be a $\K$-$\pd$-subalgebra of  $\cQ_S$ such that every nonzero divisor of $F$ is a unit of $F$. Let $H$ be a $\delta$-closed subgroup of $\Galdelta(\cQ_S/\K)$. Then, the following hold: 
\begin{itemize}
\item $\cQ_S^{\Galdelta(\cQ_S/F)}:=\{f \in \cQ_S \ |Ê\ \forall \tau \in \Galdelta(\cQ_S/F), \tau(f)=f \} =F$;
\item $\Galdelta(\cQ_S/\cQ_S^H)=H$.
\end{itemize}
\end{propo}

Let $S$ be a PPV ring over $\K$ for \eqref{eq9} and let $U \in \GL_n(S)$ be a fundamental matrix of solutions.
Then, the $\K$-$\phi$-algebra $R$ generated by the entries of $U$ and $\det(U)^{-1}$ is a PV ring
for \eqref{eq9} over $\K$ and we have $\cQ_R \subset \cQ_S$. One can identify  $\Gal^{\d}(\cQ_S/\K)$ with a subgroup of $\Gal(\cQ_R/\K)$ by restricting the elements of  $\Gal^{\d}(\cQ_S/\K)$ to $\cQ_{R}$.

\begin{propo}[\cite{HS}, Proposition 2.8]\label{propo:zarclosurePPvgaloisgroup}
The group $\Gal^{\d}(\cQ_S/\K)$ is a Zariski-dense subgroup of  $\Gal(\cQ_R/\K)$.
\end{propo}

\subsection{Hypertranscendency criteria for equations of order one}\label{hyp crit generaux}

The hypertranscendence criteria contained in \cite{HS} are stated for $\pd$-fields $\K$  such that the $\delta$-field $\k:=\K^{\phi}$ is differentially closed. Recently some schematic versions (see for instance \cite{Wib} or \cite{DVHAPac}) of \cite{HS} have been developed and allow to work over $\pd$-fields with  algebraically closed field of constants. One could use  this schematic approach to  show that  the hypertranscendence criteria of \cite{HS} still hold over $\pd$-fields with algebraically closed field of constants (not necessarily differentially closed). However, for the sake of clarity and simplicity of exposition, we prefer to show that one can deduce these criteria  directly from the ones contained in \cite{HS}, {\it via} simple descent arguments. The following result is due to Michael Singer.

\begin{propo}\label{propo:hypertransrank1}
Let $\K$ be a $\pd$-field  with $\k:=\K^\phi$ algebraically closed and let $(a,b) \in \K^{\times} \times \K$ . Let $R$ be a $\K$-$\pd$-algebra
and let $v \in R \setminus \{0\}$. \begin{itemize}
\item If $\phi(v)-v=b$ and $v$ is hyperalgebraic over $\K$, then there exist a nonzero linear homogeneous $\delta$-polynomial 
$\mathcal{L}(y) \in \k\{y\}$ and an element $f \in \K$ such that 
$$
\mathcal{L}(b)= \phi(f)-f.
$$
\item 
Assume moreover that $v$ is invertible in $R$. If $\phi(v)=av$ and if $v$ is hyperalgebraic over $\K$, then there exist a nonzero linear homogeneous $\delta$-polynomial 
$\mathcal{L}(y) \in \k\{y\}$ and an element $f \in \K$ such that 
$$
\mathcal{L}\left(\frac{\delta(a)}{a}\right)=\phi(f)-f.
$$
\end{itemize} 
The converse of either statement is true if $R^\phi=\k$.
\end{propo}

\begin{proof}
Let us prove the first statement. Let $\widetilde{\k}$ be a $\delta$-closure of $\k$. Lemma \ref{lem:extconst} assures that
$\L:=\operatorname{Frac}(\widetilde{\k}\otimes_\k \K)$ is a $\pd$-field extension of $\K$ such that $\L^\phi=\tilde{\k}$.
Let $\L\{y\}$ be the ring of $\delta$-polynomials in one variable over $\L$ endowed with the structure of $\L$-$\pd$-algebra induced by  
setting $\phi(y):= y+ b$. 
Without loss of generality,
we can assume that $R=\K\{v\}$. We identify $R$ with $\K\{y\}/\mathfrak{I}$ for some $\pd$-ideal $\mathfrak{I}$ of $\K\{y\}$.
Since $v$ is hyperalgebraic over $\K$, we have $\mathfrak{I} \neq \{0\}$.  Moreover,  we have $\mathfrak{I} \neq \K\{y\}$ because $R \neq \{0\}$. We claim that  $(\mathfrak{I}) \cap \K\{y\} =\mathfrak{I}$ where $(\mathfrak{I})$ denotes the $\pd$-ideal generated by $\mathfrak{I}$ in $\L\{y\}$. Indeed, choose   a $\K$-basis $(c_i)_{i \in I}$
of $\L$ with $c_{i_0}=1$ for some $i_{0} \in I$. Note that $(c_i)_{i \in I}$ is also a basis of the $\K\{y\}$-module $\L\{y\}$.  Then, $(\mathfrak{I})$ consists of the sums of the form $\sum a_i c_i$ with $a_i\in \mathfrak{I}$. It follows easily that $(\mathfrak{I}) \cap \K\{y\} =\mathfrak{I}$, as claimed. In particular,
$(\mathfrak{I})$ is a proper ideal of $\L\{y\}$ and, hence, is contained in some maximal $\pd$-ideal $\mathfrak{M}$ of $\L\{y\}$. The ring 
$S:=\L\{y\}/\mathfrak{M}$ is a PPV ring over $\L$ for $\phi(y)=y+ b$. The image $u$ of
$y$ in $S$ is hyperalgebraic over $\L$ (because $\mathfrak M \neq \{0\}$) and is a solution of $\phi(y)=y+ b$. By \cite[Proposition 3.1]{HS}, there exist a nonzero linear homogenous $\delta$-polynomial
$\mathcal{L}_0(y) \in \widetilde{\k}\{y\}$ and $g \in \L$ such that
\begin{equation}\label{eq:overdiffclosure}
\mathcal{L}_0(b)=\phi(g)-g.
\end{equation}
Let $(h_i)_{i \in I}$ be a $\k$-basis of $\K$.  Without loss of generality, we can assume that 
$$
\mathcal{L}_0(y)=\delta^{n+1}(y)+\sum_{i=0}^{n}  c_i \delta^{i}(y) \text{ and } 
g:=\frac{\sum_{i=1}^{r} a_i \otimes h_i}{\sum_{i=1}^{s} b_i \otimes h_i}
$$ 
where 
$a_i, b_i,c_i \in \widetilde{\k}$ and $b_{1}=1$. It is clear that the equation \eqref{eq:overdiffclosure} can be rewritten as an equation of the form  
$$
\sum_j P_j((a_i)_{i\in \{1,\dots,r\}},(b_i)_{i\in \{2,\dots,s\}},(c_i)_{i\in \{1,\dots,n\}}) \otimes h_j =0
$$
where the $P_j$ are polynomials with coefficients in $\k$. Thus, for all $j$,  
$$
P_j((a_i)_{i\in \{1,\dots,r\}},(b_i)_{i\in \{2,\dots,s\}},(c_i)_{i\in \{1,\dots,n\}})=0.
$$ 
Since $\k$ is algebraically closed, there exist $\alpha_i$, $\beta_i$, $\gamma_i \in \k$ such that, for all $j$, we have  
$$
P_j((\alpha_i)_{i\in \{1,\dots,r\}},(\beta_i)_{i\in \{2,\dots,s\}},(\gamma_i)_{i\in \{1,\dots,n\}}) =0.
$$
Set $\b_{1}:=1$. Then, we see that 
$$
{\mathcal{L}(y):=\delta^{n+1}(y)+\sum_{i=0}^{n}  \gamma_i \delta^{i}(y)}  \text{ and } 
f:=\frac{\sum_{i} \alpha_i \otimes h_i}{\sum_{i} \beta_i \otimes h_i}
$$ 
satisfy the conclusion of the first part of the proposition.

Conversely, if $R^\phi =\k$ and if there exist a nonzero linear homogeneous $\delta$-polynomial 
$\mathcal{L}(y) \in \k\{y\}$ and an element $f \in \K$ such that 
$\mathcal{L}(b)= \phi(f)-f$, then $\mathcal{L}(v)-f$ belongs to $R^\phi =\k$. Since $\mathcal{L}(y)$
is nonzero,  $v$ is differentially algebraic over $\K$.

The proof of the second statement is similar. It can also be deduced from the first statement by noticing that, if $\phi(v)=a v$ then $\phi(\frac{\delta v}{v}) =\frac{\delta v}{v}
+ \frac{\delta a}{a}$ and by using the fact that $v$ is hyperalgebraic over $\K$ if and only if the same holds for $\frac{\delta v}{v}$.

\end{proof}

\begin{rem}\label{rem:invertiblerk1}
In Proposition \ref{propo:hypertransrank1}, we require that $v$ is invertible in $R$. This assumption is automatically satisfied if 
we assume that $R$ is similar to a total PPV ring. More precisely, assume that $R= \oplus_{x \in \Z/s\Z} K_x$, where 
the $K_x$ are $\delta$-field extensions of $\K$, such that $\phi(K_x)=K_{x+\overline{1}}$. Then, any nonzero solution
$v \in R$ of $\phi(y)=ay$ for $a \in \K^\times$ is  invertible. Indeed,  $v=\sum_{x \in \Z/s\Z} v_x$ for some $v_x \in K_x$. Since $v \neq 0$, there exists $x_{0} \in \Z/s\Z$ such that $v_{x_{0}} \neq 0$.  From the equation $\phi(v)=av$, we get $\phi(v_{x_{0}-\overline{1}})=a v_{x_{0}}$. So, $v_{x_{0}-\overline{1}} \neq 0$. Iterating this argument, we see that, for all $x \in \Z/s\Z$, $v_{x} \neq 0$. Hence, $v$ is invertible in $R$. 
\end{rem}

\subsection{Isomonodromy and projective isomonodromy}

Let $\K$ be a $\pd$-field with $\k:=\K^\phi$ algebraically closed. Let  $\widetilde{\k}$ be a $\delta$-closure of $\k$. Let $\Const:=\widetilde{\k}^\delta$ be the (algebraically closed) field of constants of $\widetilde{\k}$. Lemma \ref{lem:extconst} ensures that  
$\widetilde{\k}\otimes_\k \K$ is an integral domain and that  
${\L:=\operatorname{Frac}(\widetilde{\k}\otimes_\k \K)}$ is a $(\phi,\delta)$-field extension of $\K$ such that $\L^{\phi}=\widetilde{\k}$. We let $\cQ_S$ be the total ring of quotients of a PPV ring  $S$ over $\L$ of the difference system 
$$
\phi(Y)=AY
$$ 
where $A \in \GL_n(\K)$.

\begin{propo}\label{propo:ppvcompintiso}
 The following statements are equivalent:
\begin{enumerate}
\item the group $\Galdelta(\cQ_S/\L)$ is conjugate to a subgroup of $\GL_n(\Const)$;
\item there exists $\widetilde{B} \in \L^{n \times n}$ such that
\begin{equation}\label{eq:compintconstgroupdiffclos}
\phi(\widetilde{B})=A\widetilde{B}A^{-1}+ \delta(A)A^{-1};
\end{equation}
\item there exists $B \in \K^{n \times n}$ such that
$$
\phi(B)=ABA^{-1}+ \delta(A)A^{-1}.
$$
\end{enumerate}
\end{propo}
\begin{proof}
The equivalence between (1) and (2) is \cite[Proposition 2.9]{HS}. In order to complete the proof of the proposition, it remains to prove that, if the equation \eqref{eq:compintconstgroupdiffclos} has a solution $\widetilde{B}$ in $\L^{n \times n}$, then it has a solution in $\K^{n \times n}$. This follows from an argument similar to the descent argument used in the proof of Proposition \ref{propo:hypertransrank1}.
\end{proof}

 We now consider a ``projective isomonodromic'' situation, in the spirit of \cite{MitschSingProjiso}. Let $U \in \GL_n(S)$ be a fundamental 
matrix of solutions of $\phi(Y)=AY$ and let $d:=\det(U) \in S^{\times}$.

\begin{propo}\label{propo:projisomonoundgencondition}
Assume that the difference Galois group of $\phi(Y)=AY$ over the $\phi$-field $\K$ contains $\SL_{n}(\k)$ and that the parametrized difference Galois group of $\phi (y)= \det(A) y$ over the $\pd$-field $\L$ is included in $\Const^{\times}$.
Then, we have the following alternative:
\begin{enumerate}
\item $\Galdelta(\cQ_S/\L)$ is conjugate to a subgroup of $\GL_{n}(\Const)$ that contains $\SL_{n}(\Const)$; 
\item $\Galdelta(\cQ_S/\L)$ is equal to a subgroup of $\Const^{\times} \SL_{n}(\widetilde{\k})$ that contains $\SL_n(\widetilde{\k})$.
\end{enumerate}
Moreover, the first case holds if and only if there exists $B \in \K^{n \times n}$ such that 
\begin{equation}\label{gen integ hypertranalgclos}
\phi(B)=ABA^{-1}+\delta(A)A^{-1}.
\end{equation}
\end{propo}

\begin{proof}
Let $R$ be the $\L$-$\phi$-algebra generated by the entries of $U$ and by $\det(U)^{-1}$; this is a PV ring for $\phi(Y)=AY$ over the $\phi$-field $\L$.
Using \cite[Corollary 2.5]{CHS}, we see that the hypothesis that the difference Galois group of $\phi(Y)=AY$ over the $\phi$-field $\K$ contains $\SL_{n}(\k)$ implies that 
$\Gal(\cQ_R/\L)$ contains $\SL_n(\widetilde{\k})$. So, $\Gal(\cQ_R/\L)^{der}=\SL_n(\widetilde{\k})$. Since $\Galdelta(\cQ_S/\L)$ is Zariski-dense in $\Gal(\cQ_R/\L)$, we have that $\Galdelta(\cQ_S/\L)^{\derKol}$ (this is the Kolchin-closure of the derived subgroup of $\Galdelta(\cQ_S/\L)$; see Section~\ref{subsubsec:preliminarydiffderived group}) is Zariski-dense in $\Gal(\cQ_R/\L)^{der}=\SL_n(\widetilde{\k})$. By Proposition~\ref{propo3}, $\Galdelta(\cQ_S/\L)^{\derKol}$ is  either conjugate to $\SL_{n}(\Const)$ or equal to $\SL_{n}(\widetilde{\k})$. Since $\Galdelta(\cQ_S/\L)^{\derKol}$ is a normal subgroup of $\Galdelta(\cQ_S/\L)$, Lemma \ref{lem normalisateur} ensures that $\Galdelta(\cQ_S/\L)$ is either conjugate to a subgroup of $\widetilde{\k}^{\times} \SL_{n}(\Const)$ containing $\SL_{n}(\Const)$ or is equal to a subgroup of $\GL_{n}(\widetilde{\k})$ containing $\SL_{n}(\widetilde{\k})$. But, the parametrized difference Galois group of $\phi (y)= \det(A) y$ over $\L$, which can be identified with $\det(\Galdelta(\cQ_S/\L))$, is contained in $\Const^{\times}$. Therefore, $\Galdelta(\cQ_S/\L)$ is either contained in $\Const^{\times} \SL_{n}(\Const)=\GL_{n}(\Const)$ or in $\Const^{\times} \SL_{n}(\widetilde{\k})$. Whence the first part of the proposition.

The second part of the proposition follows from Proposition \ref{propo:ppvcompintiso}. 
\end{proof}

\begin{propo}\label{propo:projisomonodgencondition}
Assume that the difference Galois group of $\phi(Y)=AY$ over the $\phi$-field $\K$ contains $\SL_{n}(\k)$ and that the parametrized difference Galois group of $\phi (y)= \det(A) y$ over the $\pd$-field $\L$ is $\widetilde{\k}^{\times}$.
Then, we have the following alternative:
\begin{enumerate}
\item $\Galdelta(\cQ_S/\L)$ is conjugate to $\widetilde{\k}^{\times}\SL_{n}(\Const)$; 
\item $\Galdelta(\cQ_S/\L)$ is equal to $\GL_{n}(\widetilde{\k})$.
\end{enumerate}
Moreover, the first case holds if and only if there exists $B \in \K^{n \times n}$ such that 
\begin{equation}\label{gen integ hypertranalgclos}
\phi(B)=ABA^{-1}+\delta(A)A^{-1}-\frac{1}{n}\delta(\det(A))\det(A)^{-1}I_{n}.
\end{equation}
\end{propo}

\begin{proof}
Arguing as for the proof of Proposition~\ref{propo:projisomonoundgencondition}, we see that $\Galdelta(\cQ_S/\L)$ is either conjugate to a subgroup of $\widetilde{\k}^{\times} \SL_{n}(\Const)$ containing $\SL_{n}(\Const)$ or equal to a subgroup of $\GL_{n}(\widetilde{\k})$ containing $\SL_{n}(\widetilde{\k})$. Now, the first part of the proposition follows from the fact that the parametrized difference Galois group of $\phi (y)= \det(A) y$ over $\L$, which can be identified with $\det(\Galdelta(\cQ_S/\L))$, is equal to $\widetilde{\k}^{\times}$.

We shall now prove that the first case holds if and only if there exists $B \in  \L^{n \times n}$ such that 
\begin{equation}\label{gen integ hypertrandiffclos}
\phi(B)=ABA^{-1}+\delta(A)A^{-1}-\frac{1}{n}\delta(\det(A))\det(A)^{-1}I_{n}.
\end{equation}

Let us first assume that $\Galdelta(\cQ_S/\L)$ is conjugate to $\widetilde{\k}^{\times}\SL_{n}(\Const)$. So, there exists a fundamental matrix of solutions $U \in \GL_{n}(S)$ of $\phi(Y)=AY$ such that, for all $\sigma \in \Galdelta(\cQ_S/\L)$, there exist $\rho_{\sigma} \in \widetilde{\k}^{\times}$ and $M_{\sigma} \in \SL_{n}(\Const)$ such that 
$
\sigma(U)=U\rho_{\sigma}M_{\sigma}.
$ 
Note that  
$
\sigma(d)=d\rho_{\sigma}^{n}.
$
Easy calculations show that the matrix  
$$
B:=\delta(U)U^{-1} - \frac{1}{n} \delta(d)d^{-1}I_{n} \in S^{n \times n}
$$ 
is left invariant by the action of $\Galdelta(\cQ_S/\L)$, and, hence, belongs to $\L^{n \times n}$ in virtue of Proposition \ref{propo:ppvcorres}, and that $B$ satisfies equation \eqref{gen integ hypertrandiffclos}.

Conversely, assume that there exists $B \in \L^{n \times n} $ satisfying equation \eqref{gen integ hypertrandiffclos}.
Consider 
$$
B_{1}=B+\frac{1}{n} \delta(d)d^{-1} I_{n} \in S^{n\times n}.
$$ 
Note that
$$
\phi(B_{1})=AB_{1}A^{-1}+\delta(A)A^{-1}.
$$
Let $U \in \GL_{n}(S)$ be a fundamental matrix of solutions of $\phi Y=AY$. We have 
$
\phi(\delta(U)-B_{1}U)
=A(\delta(U)-B_{1}U).
$
So, there exists $C \in  \widetilde{\k}^{n \times n}$ such that ${\delta(U)-B_{1}U=UC}$. Since $\widetilde{\k}$ is differentially closed, we can find $D \in \GL_{n}(\widetilde{\k})$ such that $\delta(D)+CD=0$. Then, $V:=UD$ is a fundamental matrix of solutions of $\phi Y=AY$ such that 
$
\delta(V)= 
B_{1}V.
$
Consider $\sigma \in \Galdelta(\cQ_S/\L)$ and let $M_{\sigma} \in \GL_{n}(\widetilde \k)$ be such that $\sigma(V)=VM_{\sigma}$; note that $\sigma(d)=d \rho_{\sigma}$ where $\rho_{\sigma}=\det(M_{\sigma})$. On the one hand, we have $\sigma(\delta(V))=\sigma(B_{1}V)=(B_{1}+\frac{1}{n}\delta(\rho_{\sigma})\rho_{\sigma}^{-1} I_{n}) VM_{\sigma}$. On the other hand, we have $\sigma(\delta(V))=\delta(\sigma(V))=\delta(VM_{\sigma})=B_{1}VM_{\sigma}+V\delta(M_{\sigma})$. So, $\frac{1}{n}\delta(\rho_{\sigma})\rho_{\sigma}^{-1}M_{\sigma}=\delta(M_{\sigma})$. So, the entries of $M_{\sigma}=(m_{i,j})_{1 \leq i,j \leq n}$ are solutions of $\delta(y)=\frac{1}{n}\delta(\rho_{\sigma})\rho_{\sigma}^{-1} y$. Let $i_{0},j_{0}$ be such that $m_{i_{0},j_{0}} \neq 0$. Then, $M_{\sigma}=m_{i_{0},j_{0}} M'$ with $M'=\frac{1}{m_{i_{0},j_{0}}} M_{\sigma} \in \GL_{n}(\widetilde \k^{\delta})=\GL_{n}(\Const)$, whence the desired result.

To conclude the proof, we have to show that if \eqref{gen integ hypertrandiffclos} has a solution $B$ in $\L^{n\times n}$ then it has a solution in $\K^{n \times n}$. This can be proved by using an argument similar to the descent argument used in the proof of Proposition \ref{propo:hypertransrank1}. 
\end{proof}

\section{Hypertranscendence of solutions of Mahler equations}\label{sec hyptr}

Now, we focus our attention on Mahler equations. 

We use the notation of Section \ref{mahler as diff eq}: $p\geq 2$ is an integer, $\K:=\cup_{j \geq 1} \C\left(z^{1/j}\right)$ and $\K':=\K(\log(z))$. 
We endow $\K$ with the structure of $\phi$-field given by $\phi (f(z)):=f(z^{p})$. We endow $\K':=\K(\log(z))$ with the structure of $\phi$-field given by $\phi (f(z, \log(z))):=f(z^{p}, p\log(z))$. We have $\K^{\phi}=\K'^\phi=\C$. 

The derivation 
$$
\delta:=z \log(z) \frac{d}{dz}
$$ 
gives a structure of $\pd$-field over $\K'$ (so, $\delta$ commutes with $\phi$, and this is the reason why we work with $\delta$ instead of a simplest derivation). We also set 
$$
\vartheta:= z \frac{d}{dz}.
$$ 

We let $\widetilde{\C}$ denote a differential closure of $(\C,\delta)$.  We have $\widetilde{\C}^\delta=\C$.
As in Lemma \ref{lem:extconst},  we consider $\L= \operatorname{Frac}(\widetilde{\C} \otimes_\C \K')=\cup_{j \geq 1}\widetilde{\C}\left(z^{1/j}\right)(\log(z))$, which is a $\pd$-field extension of $\K'$ such that
   $\L^\phi=\widetilde{\C}$.


\subsection{Homogeneous  Mahler  equations of order one}\label{sec3}

In this section, we consider the difference equation of order one 
\begin{equation}\label{eq ordre un}
\phi(y)=ay
\end{equation}
where $a \in \C(z)^{\times}$. 
We let $S$ be a PPV ring over $\L$ for the equation (\ref{eq ordre un}).

Since $S$ is an $\L$-$\pd$-algebra, it can be seen as a $\C(z)$-$\vartheta$-algebra ({\it i.e.}, over the differential field $(\C(z),\vartheta)$) by letting $\vartheta$ acts as $\frac{1}{\log(z)} \delta$.

\begin{propo}\label{propo:mahlerrang1homoge}
Let  $R$ be a $\K'$-$\pd$-algebra such that $R^\phi=\C$. 
 Let $u$ be an invertible element of $R$ such that $\phi(u)=au$. The following statements are equivalent:
 
\begin{enumerate}
\item  $u$ is hyperalgebraic over $(\C(z),\vartheta)$\footnote{Of course, $u$ is hyperalgebraic over $(\C(z),\vartheta)$ if and only if $u$ is hyperalgebraic over $(\C(z),d/dz)$.};
\item $\Galdelta(\cQ_S/\L)$ is conjugate to a subgroup of $\C^{\times}$;
\item there exists $d\in \C(z)$ such that 
$\vartheta(a)=a(p \phi(d)-d)$;

\item there exist $ c \in \C^{\times}$, $m \in \Z$ and $f \in \C(z)^{\times}$ such that $a=cz^{m}\frac{\phi(f)}{f}$.

\end{enumerate}
\end{propo}

\begin{proof}
We first prove the implication $(3) \Rightarrow (2)$. Assume that there exists ${d\in \C(z)}$ such that 
$
\vartheta (a)=a(p \phi(d)-d).
$  
Then, $d_{1} :=d \log (z) \in \K'$ satisfies 
$
\delta(a)=a(\phi(d_{1})-d_{1})
$ 
and, hence,  $\Galdelta(\cQ_{S}/\L)$ is conjugate to a subgroup of $\C^{\times}$ in virtue of Proposition~\ref{propo:ppvcompintiso}. 

We now prove $(2) \Rightarrow (3)$.  We assume that $\Galdelta(\cQ_{S}/\L)$ is conjugate to a subgroup of $\C^{\times}$. By Proposition \ref{propo:ppvcompintiso}, there exists $d_{1}\in \K'$ such that 
$
\delta(a)=a(\phi (d_{1})-d_{1}).
$ 
Therefore, we have 
\begin{equation}\label{zddz}
\vartheta(a)=a(p\phi (d_{2})-d_{2})
\end{equation}
with $d_{2} :=\frac{d_{1}}{\log(z)} \in \K'$. We shall now prove that there exists $d_{3} \in \K$ such that 
$
\vartheta(a)=a(p\phi (d_{3})-d_{3}).
$ 
Indeed, let $u(X,Y) \in \C(X,Y)$, $k \geq 1$ and  $v \in \C(X)$ be such that 
$
d_{2}=u(z^{1/k},\log(z)) \text{ and } \frac{\vartheta(a)}{a}=v(z).
$ 
The equation (\ref{zddz}) can be rewritten as follows 
$$
v(z)=pu(z^{p/k},p\log(z))-u(z^{1/k},\log(z)).
$$
Since $z^{1/k}$ and $\log(z)$ are algebraically independent over $\C$, we get 
$$
v(X^{k})=pu(X^{p},pY)-u(X,Y).
$$
We see $u(X,Y)$ as an element of $\C(X)(Y) \subset \C(X)((Y))$ as follows:
$
{u(X,Y)=\sum_{j \geq -N} u_{j}(X)Y^{j},}
$ 
for some $N\in \Z$. We have 
$$
v(X^{k})=pu(X^{p},pY)-u(X,Y)=\sum_{j \geq -N} (p^{j+1}u_{j}(X^{p})-u_{j}(X))Y^{j}.
$$
Equating the coefficients of $Y^{0}$ in this equality, we obtain  
$$
pu_{0}(X^{p})-u_{0}(X)=v(X^{k}).
$$
Hence, $d_{3}:=u_{0}(z^{1/k})$ has the required property. 

We claim that $d_{3}$ belongs to $\C(z)$. Indeed, suppose to the contrary that ${d_{3} \not \in \C(z)}$. Let $k \geq 2$ be such that $d_{3} \in \C(z^{1/k})$. We see $d_{3}$ in $\C((z^{1/k}))$:
$
d_{3}=\sum_{j \geq -N} d_{3,j}z^{j/k}
$
for some $N \in \Z$. Let $j_{0} \in \Z$ be such that $k \not \vert j_{0}$ and $d_{3,j_{0}} \neq 0$, with $\vert j_{0} \vert$ minimal for this property. Then, the coefficient of $z^{j_{0}/k}$ in $p\phi(d_{3})-d_{3}$ is nonzero, and this contradicts the fact that $p\phi(d_{3})-d_{3}$ belongs to $\C(z)$. This proves (3).

 We now prove $(3) \Rightarrow (1)$. Assume that there exists $d\in \C(z)$ such that 
$
\vartheta (a)=a(p \phi(d)-d).
$  
Then, $d_{1} :=d \log (z) \in \K'$ satisfies 
$
\delta(a)=a(\phi(d_{1})-d_{1})
$ 
and, hence, Proposition \ref{propo:hypertransrank1} ensures that $u$ is hyperalgebraic over $(\K',\delta)$. Therefore, $u$  is hyperalgebraic over $(\K',\vartheta)$ and the conclusion follows from the fact that $(\K',\vartheta)$ is hyperalgebraic over $(\C(z),\vartheta)$. 

We now prove $(1) \Rightarrow (3)$. Proposition \ref{propo:hypertransrank1}, applied to the difference equation $\phi(y)=ay$ over the  $\pd$-field $\K'$, ensures that 
there exist
$
\mathcal{L}_{1}:= \sum_{i=1}^{\nu} \b_{i}\delta^{i}
$
with coefficients $\b_{1},\dots,\b_{\nu}=1$ in $\C$ and $g_{1} \in \C(z^{1/k},\log(z))$ such that  
\begin{equation}\label{L1}
\mathcal{L}_{1}\left(\frac{\delta (a)}{a} \right)=\phi(g_{1})-g_{1}.
\end{equation}
We shall now prove that there exists $g_{2} \in \C(z^{1/k})$ such that  
$$
\vartheta^{\nu}\left(\frac{\vartheta (a)}{a} \right)=p^{\nu+1} \phi(g_{2}) -g_{2}.
$$
Indeed, it is easily seen that there exists $v(X,Y) \in \C(X)[Y]$ such that ${\mathcal{L}_{1}\left(\frac{\delta (a)}{a} \right)=v(z,\log(z))}$. Using the fact that  
$
\delta^{i}=\log(z)^{i}\vartheta^{i} +$ terms of lower degree in $\log(z)$, we see that  
\begin{multline*}
\mathcal{L}_{1}\left(\frac{\delta (a)}{a} \right)
= \mathcal{L}_{1}\left(\log(z)\frac{\vartheta(a)}{a} \right)
=  \vartheta^{\nu}\left(\frac{\vartheta (a)}{a} \right) (\log(z))^{\nu+1} \\ 
+  \text{ terms of lower degree in $\log(z)$.}
\end{multline*}
On the other hand, let $u(X,Y) \in \C(X,Y)$ and $k \geq 1$ be such that ${g_{1}=u(z^{1/k},\log(z))}$. The equation (\ref{L1}) can be rewritten as follows 
$$
v(z,\log(z))=u(z^{p/k},p\log(z))-u(z^{1/k},\log(z)).
$$
Since $z^{1/k}$ and $\log(z)$ are algebraically independent over $\C$, we get 
$$
v(X^{k},Y)=u(X^{p},pY)-u(X,Y).
$$
We see $u(X,Y)$ as an element of $\C(X)(Y) \subset \C(X)((Y))$ as follows:
$
{u(X,Y)=\sum_{j \geq -N} u_{j}(X)Y^{j}}
$ 
for some $N\in \Z$.
So, 
$$
v(X^{k},Y)=u(X^{p},pY)-u(X,Y)=\sum_{j \geq -N} (p^{j}u_{j}(X^{p})-u_{j}(X))Y^{j}.
$$
Equating the coefficients of $Y^{\nu+1}$ in this equality, and letting $X=z^{1/k}$, we obtain, 
$$
p^{\nu+1} u_{\nu +1}(z^{p/k})-u_{\nu+1}(z^{1/k})=\vartheta^{\nu}\left(\frac{\vartheta (a)}{a} \right).
$$
Therefore, $g_{2} = u_{\nu +1}(z^{1/k}) \in \C(z^{1/k})$ has the required property. 
One can show that $g_{2}$ belongs to $\C(z)$ by arguing as for the proof of the fact that $d_{3} \in \C(z)$ in the proof of (2) $\Rightarrow (3)$ above.
We now claim that there exists $g_{3} \in \C(z)$ such that 
$$
\frac{\vartheta (a)}{a} = p \phi(g_{3}) - g_{3}.
$$
If $\nu=0$, then $g_{3}:=g_{2}$ has the expected property. Assume that $\nu >0$.
Let $G_{2} = \int \frac{g_{2}}{z}$ be some primitive of $\frac{g_{2}}{z}$ that we see as a function on some interval $(0,\epsilon)$, $\epsilon >0$. We have 
$$
\vartheta\left(p^{\nu} \phi(G_{2}) -G_{2} \right)
=p^{\nu+1} \phi(g_{2}) -g_{2}=\vartheta^{\nu}\left(\frac{\vartheta (a)}{a} \right), 
$$
so there exists $C \in \C$ such that 
$$
p^{\nu} \phi(G_{2}) -G_{2} = \vartheta^{\nu-1}\left(\frac{\vartheta (a)}{a} \right) +C.
$$
Hence, $G_{3}:=G_{2}-\frac{C}{p^{\nu}-1}$ satisfies 
$$
p^{\nu} \phi(G_{3}) -G_{3} = \vartheta^{\nu-1}\left(\frac{\vartheta (a)}{a} \right).
$$
But $G_{3} = G_{4}+\ell$ where $G_{4} \in \C(z)$ and $\ell$ is a $\C$-linear combination of $\log(z)$ and of functions of the form $\log(1-z\xi)$\footnote{Here, $\log(z)$ is the principal determination of the logarithm, and $\log(1-z\xi)$ is such that $\log(1-0\xi)=0$} with $\xi \in \C ^{\times}$. Using the $\C$-linear independence of any $\C$-linear combination of $\log(z)$ and of functions of the form $\log(1-z\xi)$ with $\xi \in \C ^{\times}$ with any element of $\C(z)$, we see that the equality 
$$
p^{\nu} \phi(G_{3}) -G_{3} = \left(p^{\nu} \phi(G_{4}) -G_{4}\right) + \left(p^{\nu} \phi(\ell) -\ell\right) = \vartheta^{\nu-1}\left(\frac{\vartheta (a)}{a} \right)
$$
implies that 
$$
p^{\nu} \phi(G_{4}) -G_{4} = \vartheta^{\nu-1}\left(\frac{\vartheta (a)}{a} \right).
$$
Iterating this argument, we find $g_{3} \in \C(z)$ with the expected property. This proves (3).

We shall now prove (3) $\Rightarrow$ (4). We assume that there exists $d\in \C(z)$ such that 
$\vartheta(a)=a(p \phi(d)-d)$. 
We write $a=cz^{m}l$ with $c \in \C^{\times}$, $m \in \Z$ and $l \in \C(z)$ without pole at $0$ and such that $l(0)=1$. Since $\frac{\vartheta(a)}{a}= \frac{\vartheta(c^{-1}a)}{c^{-1}a}$, we can assume that $c=1$. A fundamental solution of $\phi(y)=ay$ is given by 
$$
f_{0}=z^{\frac{m}{p-1}}\prod_{k\geq 0} \phi^{k}(l)^{-1} \in z^{\frac{m}{p-1}} \C[[z]] \subset \C((z^{\frac{1}{p-1}})).
$$ 
We have $\delta(a)a^{-1}=\phi(\widetilde{d})-\widetilde{d}$ with $\widetilde{d}=\log(z)d$. This is the integrability condition for the system of equations 
$$
\begin{cases}
\phi(y)=ay \\ 
\delta(y)=\widetilde{d}y, \text{ {\it i.e.}, } \vartheta(y)=dy.
\end{cases}
$$ 
A straightforward calculation shows that $\delta(f_{0})-\widetilde{d}f_{0}$ is a solution of $\phi(y)=ay$ so there exists $q \in \C$ such that $\delta(f_{0})=(q+\widetilde{d})f_{0}$, {\it i.e.}, $\log(z) \vartheta(f_{0}) = (q+\log(z) d)f_{0}$ (here, we work in the $(\phi,\delta)$-field $\C((z^{\frac{1}{p-1}}))(\log(z))$ and we have used the fact that the field of $\phi$-constants of $\C((z^{\frac{1}{p-1}}))(\log(z))$ is equal to $\C$, so that the solutions of $\phi(y)=ay$ in $\C((z^{\frac{1}{p-1}}))(\log(z))$ are of the form $\lambda f_{0}$ for some $\lambda \in \C$). Therefore, $\vartheta(f_{0})=d f_{0}$. So, $f_{0}$ satisfies a nonzero linear differential equation with coefficients in $\K$, and also a nonzero linear Mahler equation with coefficients in $\K$. It follows from \cite[Theorem 1.3]{Be94} that $f_{0} \in \C(z^{\frac{1}{p-1}})$. Therefore, $f_{0}=z^{\frac{m}{p-1}} h$ for some $h \in \C(z)$, and, hence, $a=\phi(f_{0})f_{0}^{-1}=z^{m} \phi(h)h^{-1}$.

We shall now prove (4) $\Rightarrow$ (3). We assume that there exists $ c \in \C^{\times}$, $m \in \Z$ and $f \in \C(z)^{\times}$ such that $a=cz^{m}\frac{\phi(f)}{f}$. Then, $\vartheta(a)/a=p\phi(d)-d$ with ${d=m/(p-1)+\frac{\vartheta(f)}{f} \in \C(z)}$. Whence the desired result. 
\end{proof}


\begin{rem}
The technics employed above could also be used in order to recover a famous result of Nishioka about the hypertranscendence of solutions of inhomogeneous Mahler equations of order one~\cite{Nish}. A Galoisian approach (but without parametrized Picard-Vessiot theory) of the work of Nishioka has been proposed by  Nguyen in \cite{NG}.
\end{rem}


\subsection{Mahler equations of higher order with large classical difference Galois group}\label{sec:hyperrank2} 

Consider the  difference  system
\begin{equation}\label{eq102}
\phi (Y)=AY
\end{equation}
with $A \in \GL_n(\C(z))$. We let $S$ be a PPV ring for \eqref{eq102} over $\L$. The aim of the present section is to study the parametrized difference Galois group $\Galdelta(\cQ_S/\L)$ of \eqref{eq102} over $\L$ under the following assumption.

\begin{hypothese}\label{hypothese gros groupe}
In the rest of this section, we assume that the difference Galois group of \eqref{eq102} over the $\phi$-field $\K$ contains $\SL_{n}(\C)$.
\end{hypothese}

Note the following result. 
 
 \begin{lem}\label{lemm:preservationgroupbaseextlog}
 Assume that the assumption (\ref{hypothese gros groupe}) holds. Then, the  difference Galois group of \eqref{eq102} over the $\phi$-field $\L$ contains $\SL_{n}(\widetilde{\C})$.
  \end{lem}
\begin{proof}
 
Corollary \ref{coro contient SL} ensures that the difference Galois group of \eqref{eq102} over the $\phi$-field $\K'$ contains $\SL_n(\C)$.
 The fact that the difference Galois group of \eqref{eq102} over the $\phi$-field $\L$ contains $\SL_{n}(\widetilde{\C})$ is now a direct consequence of \cite[Corollary~2.5]{CHS}. 

\end{proof}

Let $U \in \GL_{n}(S)$ be a fundamental matrix of solutions of \eqref{eq102} and set 

$$
d:=\det(U) \in S^\times.
$$ 
Then, $d$ is a fundamental solution of $\phi(y)=\det(A) y$ in $S$.
 We split our study of $\Galdelta(\cQ_S/\L)$ in two cases, depending on whether $d$ is hyperalgebraic or hypertranscendental over $(\L,\delta)$. Note that Proposition \ref{propo:mahlerrang1homoge} may be used to check whether $d$ is hyperalgebraic or not.

\subsubsection{Hyperalgebraic determinant}\label{subsubsec:hyperalgdet}

This section is devoted to the proof of the following result. 

\begin{theo}\label{theo hyper alg}
Assume that the assumption (\ref{hypothese gros groupe}) holds and that $d$ is hyperalgebraic over $(\C(z),\vartheta)$ (or, equivalently, that the parametrized difference Galois group of $\phi(y)=\det (A) y$ over $\L$ is included in $\C^{\times}$; see Proposition~\ref{propo:mahlerrang1homoge}). Then, the parametrized difference Galois group $\Galdelta(\cQ_S/\L)$ is a subgroup of $\C^{\times}\SL_{n}(\widetilde \C)$ 
containing $\SL_{n}(\widetilde \C)$.
\end{theo}

Before proceeding with the proof of this theorem, we give some lemmas.

\begin{lem}\label{propo:alternative hyperalg}
Assume that the assumption (\ref{hypothese gros groupe}) holds and that $d$ is hyperalgebraic over $(\C(z),\vartheta)$ (or, equivalently, that the parametrized difference Galois group of $\phi(y)=\det (A) y$ over $\L$ is included in $\C^{\times}$; see Proposition~\ref{propo:mahlerrang1homoge}). Then, we have the following alternative:
\begin{enumerate}
\item $\Galdelta(\cQ_S/\L)$ is conjugate to a subgroup of $\GL_{n}(\C)$ containing $\SL_{n}(\C)$; 
\item $\Galdelta(\cQ_S/\L)$ is equal to a subgroup of $\C^{\times}\SL_{n}(\widetilde \C)$ containing $\SL_{n}(\widetilde \C)$.
\end{enumerate}
Moreover, the first case holds if and only if there exists $B \in \K^{n \times n}$ such that 
\begin{equation}\label{integ hyperalg}
p\phi(B)=ABA^{-1}+\vartheta(A)A^{-1}.
\end{equation}
\end{lem}

\begin{proof}
Using Proposition \ref{propo:projisomonoundgencondition}, we are reduced to prove that the equation 
\begin{equation}\label{integ hyperalg sur L}
\phi(B)=ABA^{-1}+\delta(A)A^{-1}
\end{equation}
has a solution $B \in \K'^{n \times n} $ if and only if the equation (\ref{integ hyperalg}) has a solution $B \in \K^{n \times n}$.

Assume that the equation (\ref{integ hyperalg sur L}) has a solution $B \in \K'^{n \times n}$. Let ${u(X,Y) \in \C(X,Y)^{n\times n}}$, $k \geq 1$,  $v(X) \in \GL_{n}(\C(X))$ and $w(X) \in \C(X)^{n\times n}$ be such that 
$$
B=u(z^{1/k},\log(z)), \ A=v(z) \text{ and } \delta(A)A^{-1}=\log(z)w(z).
$$ 
The equation (\ref{integ hyperalg sur L}) can be rewritten as follows 
$$
u(z^{p/k},p\log(z))=v(z)u(z^{1/k},\log(z))v(z)^{-1}+\log(z)w(z).
$$
Since $z^{1/k}$ and $\log(z)$ are algebraically independent over $\C$, we get 
$$
u(X^{p},pY)=v(X^{k})u(X,Y)v(X^{k})^{-1}+Yw(X^{k}).
$$
We see $u(X,Y)$ as an element of $\C(X)(Y)^{n \times n} \subset \C(X)((Y))^{n \times n}$:
$
u(X,Y)=\sum_{j \geq -N} u_{j}(X)Y^{j}
$ 
for some $N \in \Z$. 
We have 
$$
\sum_{j \geq -N} u_{j}(X^{p})p^{j}Y^{j}= \left(\sum_{j \geq -N}  v(X^{k})u_{j}(X)v(X^{k})^{-1}Y^{j}\right)+Yw(X^{k}).
$$
Equating the terms of degree $1$ in $Y$, we get 
$$
pu_{1}(X^{p})= v(X^{k})u_{1}(X)v(X^{k})^{-1}+w(X^{k}).
$$
Therefore, $B_{1}:=u_{1}(z^{1/k}) \in \K$ is a solution of (\ref{integ hyperalg}).   

Conversely, assume that the equation (\ref{integ hyperalg}) has a solution $B \in \K^{n \times n}$. Then $B_{1}:=B\log(z) \in \K'^{n \times n}$ satisfies 
$$
\phi(B_{1})=AB_{1}A^{-1}+\delta(A)A^{-1}.
$$

\end{proof}

\begin{lem}\label{lem:ppvformalsol}
Assume that the system $\phi(Y)=BY$, with $B \in \GL_n(\K')$, has a solution $u=(u_{1},\dots,u_{n})^{t}$ with coefficients in $\C((z^{1/k}))$ for some integer $k \geq 1$. Then, there exists a
PPV ring $T$ over $\L$ of $\phi(Y)=BY$ that contains the $\L$-$\delta$-algebra $\L\{ u_{1},\dots,u_{n} \}$. 

\end{lem}

\begin{proof}
The result is obvious if $u=(0,\dots,0)^{t}$. We shall now assume that ${u \neq (0,\dots,0)^{t}}$.
We consider the field $\widehat{\K'}:=\cup_{j \geq 1}\C((z^{1/j}))(\log(z))$. We equip $\widehat{\K'}$ with the structure of $(\phi,\delta)$-field given by $\phi(f(z,\log(z)))=f(z^{p},p \log(z))$ and $\delta = \log(z) z \frac{d}{dz}$. It is easily seen that $\widehat{\K'}^{\phi}=\C$. One can see $\K'$ as a $(\phi,\delta)$-subfield of $\widehat{\K'}$. 
We let $F=\K'\langle u_{1},\dots,u_{n} \rangle$ be the $\delta$-subfield of $\widehat{\K'}$ generated over $\K'$ by $u_{1},\dots,u_{n}$; this is a $(\phi,\delta)$-subfield of $\widehat{\K'}$ such that $F^{\phi}=\C$. By Lemma~\ref{lem:extconst}, $\widetilde{\C} \otimes_\C F $ is an integral domain and its field of fractions $\L_{1}=\L\langle u_{1},\dots,u_{n} \rangle$ is a $(\phi,\delta)$-field such that $\L_1^\phi=\widetilde{\C}$. We consider a PPV ring $S_{1}$ for $\phi (Y)=BY$ over $\L_{1}$ and we let $U \in \GL_{n}(S_{1})$ be a fundamental matrix of solutions of this difference system. We can assume that the first column of $U$ is $u$.
Then, the $\L$-$(\phi,\delta)$-algebra $T$ generated by the entries of $U$ and by $\det(U)^{-1}$ contains $\L\{ u_{1},\dots,u_{n} \}$ and is a PPV ring for $\phi (Y)=BY$ over $\L$. Whence the result. 
\end{proof}

\begin{lem}\label{sol commune}
Let us consider a vector $u=(u_{1},\dots,u_{n})^{t}$ with coefficients in ${\widehat{\K}:=\cup_{j \geq 1} \C((z^{1/j}))}$ such that $\phi(u)=Bu$ for some $B \in \GL_{n}(\K)$. Assume moreover that each $u_{i}$ satisfies some nonzero linear differential equation with coefficients in $\cup_{j \geq 1} \widetilde{\C}(z^{1/j})$, with respect to the derivation $\vartheta$.  Then, the $u_{i}$ actually belong to $\K$.  
\end{lem}

\begin{proof}
According to the cyclic vector lemma, there exists $P \in \GL_{n}(\K)$ such that $Pu = (f,\phi(f),\dots,\phi^{n-1}(f))^{t}$ for some $f \in \widehat{\K}$ which is a solution of a nonzero linear Mahler equation ({\it i.e.}, a $\phi$-difference equation) of order $n$ with coefficients in $\K$. Moreover, $f$ satisfies a nonzero linear differential equation with coefficients in $\cup_{j \geq 1} \widetilde{\C}(z^{1/j})$, with respect to the derivation $\vartheta$, because it is a $\K$-linear combination of the $u_{i}$ and the $u_{i}$ themselves satisfy such equations. It follows from \cite[Theorem 1.3]{Be94} that $f$ belongs to $\K$. Hence, the entries of $u=P^{-1}(Pu)=P^{-1}(f,\phi(f),\dots,\phi^{n-1}(f))^{t}$ actually belong to $\K$, as expected.    
\end{proof}

\begin{lem}\label{lem sol form}
There exists $c \in \C^{\times}$ such that the difference system $\phi(Y)=c^{-1}A Y$ has a nonzero solution $u=(u_{1},\dots,u_{n})^{t}$ with coefficients in  $\widehat{\K}:=\cup_{j \geq 1} \C((z^{1/j}))$.
\end{lem}

\begin{proof} 
According to \cite[Section 4]{Ro15}, the system $\phi (Y)=AY$ is triangularizable over $\widehat{\K}$, {\it i.e.}, there exists $\widehat{P} \in \GL_{n}(\widehat{\K})$ such that $\phi(\widehat{P})^{-1}A\widehat{P}=:(v_{i,j})_{1 \leq i,j \leq n}$ is upper-triangular. Let $c \in \C^{\times}$, $m \in \Z$ and $l \in 1+z^{1/k}\C[[z^{1/k}]]$ be such that ${v_{1,1}=cz^{m}l}$. We consider $A_{1}=c^{-1}A \in \GL_{n}(\C(z))$. Then, the system $\phi (Y)=A_{1}Y$ has a nonzero solution with entries in $\widehat{\K }$, namely $u=(u_{1},\dots,u_{n})^{t}:=\widehat{P} (f,0,\dots,0)^{t}$ with $f:=z^{\frac{m}{p-1}} \prod_{j \geq 0} \phi^{j}(l)^{-1}$.  
\end{proof}

\begin{proof}[Proof of Theorem \ref{theo hyper alg}]

We let $c \in \C^{\times}$ and $u=(u_{1},\dots,u_{n})^{t}$ be as in Lemma~\ref{lem sol form}, and we set $A_{1}:=c^{-1}A \in \GL_{n}(\C(z))$. Thanks to Lemma~\ref{lem:ppvformalsol}, we can consider a PPV ring $S_{1}$ for $\phi(Y)=A_{1}Y$ over $\L$ that contains $\L\{ u_{1},\dots,u_{n}\}$. We let $U_{1} \in \GL_{n}(S_{1})$ be a fundamental matrix of solutions  of $\phi (Y)=A_{1}Y$ whose first column is $u$. 

We let $G$ denote the difference Galois group of $\phi (Y)=AY$ over the $\phi$-field $\K$, and we let $G^{\delta}$ denote its parametrized difference Galois group over the $\pd$-field $\L$. 
Similarly, we let $G_{1}$ denote the difference Galois group of $\phi (Y)=A_{1}Y$ over the $\phi$-ring $\K$, and we let $G_{1}^{\delta}$ denote its parametrized difference Galois group over the $\pd$-field $\L$. 

We have $G_{1}^{der}=G^{der}=\SL_{n}(\C)$, so $G_{1}$ contains $\SL_{n}(\C)$. Moreover, the parametrized difference Galois group of $\phi (y) = \det (A_{1}) y=c^{-n} \det (A) y$ over $\L$ is a subgroup of $\C^{\times}$ (because 
the parametrized difference Galois group of $\phi (y) = \det (A) y$ over $\L$ is a subgroup of $\C^{\times}$ by hypothesis, and the parametrized difference Galois group of $\phi (y) = c^{-n} y$ over $\L$ satisfies the same property). 

We claim that $G_{1}^{\delta}$ is a subgroup of $\C^{\times} \SL_{n}(\widetilde{\C})$ that contains $ \SL_{n}(\widetilde{\C})$. Indeed, according to Lemma \ref{propo:alternative hyperalg}, it is sufficient to prove that there is no $B\in \K^{n \times n}$ such that 
$\vartheta (A_{1})=p\phi (B) A_{1}-A_{1}B.$
Suppose to the contrary that such a $B$ exists. 
The equation $\vartheta (A_{1})=p\phi (B) A_{1}-A_{1}B$, which can be rewritten as $\delta (A_{1})=\phi (\log(z)B) A_{1}-A_{1}(\log(z)B)$, ensures the integrability of the system of equations 
$$
\begin{cases}
\phi (Y)=A_{1}Y\\ 
\delta (Y)= (\log(z) B) Y.
\end{cases} 
$$
So, there exists $D\in \mathrm{GL}_{n}(\widetilde{\C})$ such that $V:=U_{1}D \in \GL_{n}(S_{1})$ satisfies 
$$
\begin{cases}
\phi(V)=A_{1}V \\ 
\d(V)=(\log(z)B)V, \text{ {\it i.e.}, } \vartheta (V)=BV.
\end{cases} 
$$

Hence, we have the equalities $\vartheta(U_{1})D+U_{1}\vartheta(D)=\vartheta(U_1D)=\vartheta(V) =BV=BU_{1}D$ so $\vartheta(U_{1})=BU_{1}-U_{1}\vartheta(D)D^{-1}$. 
This formula implies that the (finite dimensional) $\cup_{j \geq 1} \widetilde{\C}(z ^{1/j})$-vector space generated by the entries of $U_{1}$ is stable by $\vartheta$. In particular, any $u_{i}$ (recall that the $u_{i}$ are the entries of the first column of $U$) satisfies a nonzero linear differential equation with coefficients in $\cup_{j \geq 1} \widetilde{\C}(z ^{1/j})$, with respect to the derivation $\vartheta$. 
It follows from Lemma~\ref{sol commune} that the $u_{i}$ belong to $\K$. Hence, the first column of $U_{1}$ is fixed by the Galois group $G_{1}$ and this contradicts the fact that $G_{1}$ contains $\SL_{n}(\C)$. 

Therefore, $(G^{\delta})^{der}=(G_{1}^{\delta})^{der}$ contains $\SL_{n}(\widetilde{\C})$. Now, the theorem follows from Lemma~\ref{propo:alternative hyperalg}. 
\end{proof}


\subsubsection{Hypertranscendental determinant}

In the case of an hypertranscendental determinant, we can reduce the computation of the parametrized difference Galois group 
to a question concerning the existence of a rational solution of a given Mahler equation as follows.

\begin{lem}\label{lem:hypertranscontainsSLdethypertrans}
Assume that the assumption (\ref{hypothese gros groupe}) holds and that $d$ is 
hypertranscendental over $(\C(z),\theta)$ (or, equivalently, that the parametrized difference Galois group of $\phi(y)=\det (A) y$ over $\L$ is equal to $\widetilde{\C}^{\times}$). Then, we have the following alternative:
\begin{enumerate}
\item $\Galdelta(\cQ_S/\L)$ is conjugate to $\widetilde{\C}^{\times}\SL_{n}(\C)$; 
\item $\Galdelta(\cQ_S/\L)$ is equal to a $\GL_{n}(\widetilde \C)$.
\end{enumerate}
Moreover, the first case holds if and only if there exists $B \in \K^{n \times n}$ such that 
\begin{equation}\label{integ hypertran}
p\phi(B)=ABA^{-1}+\vartheta(A)A^{-1}-\frac{1}{n}\vartheta(\det(A))\det(A)^{-1}I_{n}.
\end{equation}
\end{lem}

\begin{proof}
Note that $d$ is hypertranscendental over $(\L,\delta)$. Using Proposition \ref{propo:projisomonodgencondition}, it remains to prove that the equation 
\begin{equation}\label{integ hypertran sur L}
\phi(B)=ABA^{-1}+\delta(A)A^{-1}-\frac{1}{n}\delta(\det(A))\det(A)^{-1}I_{n}
\end{equation}
has a solution $B \in \K'^{n \times n} $ if and only if the equation (\ref{integ hypertran}) has a solution $B \in \K^{n \times n}$. The proof of this fact is similar to the proof of  Lemma \ref{propo:alternative hyperalg}.
\end{proof}

Unlike to the situation of Section \ref{subsubsec:hyperalgdet},  it is not completely obvious that we can bypass the search of rational solutions 
of \eqref{integ hypertran} to decide which of the two options of Lemma~\ref{lem:hypertranscontainsSLdethypertrans} is satisfied. However, we can still get directly  some 
informations on the hypertranscendence of  solutions in $\cup_{j \geq 1} \C(z ^{1/j})$ as follows.

\begin{theo}\label{theo hyper trans}
Assume that the assumption (\ref{hypothese gros groupe}) holds and that $d$ is hypertranscendental over $(\C(z),\vartheta)$. Assume that the difference system $\phi(Y)=AY$ admits a nonzero solution ${u =(u_{1},\dots,u_{n})^t}$ with coefficients in $\C((z^{1/k}))$ for some integer $k \geq 1$. Then, at least one of the $u_{i}$ is hypertranscendental over $(\C(z),\vartheta)$.
\end{theo}

Note the following immediate corollary, which is particularly interesting when one works with difference equations rather than with difference systems. 

\begin{coro}\label{coro hyper trans}
Assume that the assumption (\ref{hypothese gros groupe}) holds and that $d$ is hypertranscendental over $(\C(z),\vartheta)$. Assume that the difference system $\phi(Y)=AY$ admits a nonzero solution ${u =(f,\phi(f),\dots,\phi^{n-1}(f))^t}$ for some $f \in \C((z^{1/k}))$ and some integer $k \geq 1$. Then, $f$ (and, hence, any $\phi^{i}(f)$) is hypertranscendental over $(\C(z),\vartheta)$.
\end{coro}

The arguments employed in the proof of Theorem \ref{theo hyper trans} given below are  very similar to the ones used in the hyperalgebraic case.
 But, we need a new descent argument, that is contained in the following lemma.

 \begin{lem} \label{lem:nonforking generics}
Let $L$ be a $\delta$-field and let $L\langle a \rangle$ and $L\langle b_1,\dots,b_n \rangle$ be two $\delta$-field extensions of $L$, both contained in a same $\delta$-field extension of $L$. Assume
 that $a$ is hypertranscendental over $L$ and that any $b_i$ is hyperalgebraic over $L$. Then, the 
 field extensions $L\langle a \rangle$ and $L\langle b_1,\dots,b_n \rangle$ are linearly disjoint over $L$.
  \end{lem}
\begin{proof}
If $L\langle a \rangle$ and $L\langle b_1,\dots,b_n \rangle$ are not  linearly disjoint over $L$ then $a$ is hyperalgebraic over $L\langle b_1,\dots,b_n \rangle$. This implies that the differential transcendence degree of the field $L\langle a, b_1,\dots,b_n \rangle$ over $L\langle b_1,\dots,b_n \rangle$ is zero. Since the differential transcendence degree of $L\langle b_1,\dots,b_n \rangle$ over $L$ is also zero, by hypothesis,
we find that the differential transcendence degree of $L\langle a, b_1,\dots,b_n \rangle$ over $L$ is zero by classical properties
of the transcendence degree. This implies that $a$ is hyperalgebraic over $L$. 
\end{proof}

\begin{proof}[Proof of Theorem \ref{theo hyper trans}]

 Thanks to Lemma~\ref{lem:ppvformalsol}, we can assume that the PPV ring $S$ for $\phi(Y)=AY$ over $\L$ contains $\L\{ u_{1},\dots,u_{n}\}$. We can assume that the first column of the fundamental matrix of solutions $U \in \GL_{n}(S)$ of $\phi (Y)=AY$ is $u$. 
 
We let $G$ denote the difference Galois group of $\phi (Y)=AY$ over the $\phi$-field $\K$, and we let $G^{\delta}$ denote its parametrized difference Galois group over the $\pd$-ring $\L$.  Since $d$ is hypertranscendental over $\L$, the parametrized difference Galois group of $\phi (y) = \det (A) y$ over $\L$ is   $\widetilde{\C}^{\times}$.

Note that Lemma \ref{lem:hypertranscontainsSLdethypertrans} implies that $G^{\delta}$ is Kolchin-connected. So, $S$ is an integral domain.

We claim that at least one of the $u_i$ is hypertranscendental over $\L$.  Suppose to the  contrary that all of them are hyperalgebraic. In particular, $G^{\delta}$ is a strict subgroup of $\GL_{n}(\widetilde{\C})$. Lemma \ref{lem:hypertranscontainsSLdethypertrans} ensures that there exists $B \in \K^{n \times n}$ such that 
\begin{equation}
p\phi(B)=ABA^{-1}+\vartheta(A)A^{-1}-\frac{1}{n}\vartheta(\det(A))\det(A)^{-1}I_{n}.
\end{equation}
This equation can be rewritten as 
$$ \phi(B_0)=A B_0 A^{-1} + \delta(A)A^{-1} -\frac{1}{n}\delta(\det(A))\det(A)^{-1}I_{n},$$
where $B_0=\log(z) B$. Set $B_1:=B_0+ \frac{\delta(d)}{nd}$. Note that 
$$
\phi(B_1)=A B_1A^{-1} +\delta(A)A^{-1}.
$$
This equation ensures the integrability of the system of equations 
$$
\begin{cases}
\phi (Y)=AY\\ 
\delta (Y)=  B_1 Y.
\end{cases} 
$$
So, there exists $D\in \mathrm{GL}_{n}(\widetilde{\C})$ such that $V:=UD \in \GL_{n}(S)$ satisfies 
$$
\begin{cases}
\phi(V)=AV \\ 
\d(V)=B_1 V, \text{ {\it i.e.}, } \vartheta(V)=(B+ \frac{\vartheta(d)}{nd})V.
\end{cases} 
$$

In particular, we have $\vartheta(U)D+U\vartheta(D)=\vartheta(U_1D)=\vartheta(V) =(B  + \frac{\vartheta(d)}{nd})  UD$ so $$\vartheta(U)=  (B  + \frac{\vartheta(d)}{nd}) U-U\vartheta(D)D^{-1}.$$ 

If we set $F=\cup_{j \geq 1} \widetilde{\C}(z ^{1/j})$, the previous formula implies that the  $F\langle d\rangle$\footnote{Here, $F\langle d \rangle$ denotes the $\vartheta$-field extension generated by $d$ over $F$.}-vector subspace of $\mathcal{Q}_{S}$ generated by the entries of $U$ and all their successive $\vartheta$-derivatives is of finite dimension. In particular, any $u_{i}$ satisfies a nonzero linear differential equation $\mathcal{L}_i(y)=0$ with coefficients in $F\langle d\rangle$, with respect to the derivation $\vartheta$. We can assume that the coefficients of $\mathcal{L}_i(y)$ belong to $F\{ d\}$. We write $\mathcal{L}_i(y) =\sum_\alpha L_{i,\alpha}(y) d_\alpha$ where $L_{i,\alpha}(y)$ is a linear differential operator with coefficients in $F$, with respect to the derivation $\vartheta$,
and $d_\alpha$ is a monomial in the $\vartheta^i(d)$'s. By Lemma~\ref{lem:nonforking generics},  the $\vartheta$-fields $F\langle d\rangle$ and $F\langle u_{1},...,u_{n}\rangle$ are linearly disjoint over $F$. It follows easily that there exists some nonzero $L_{i,\alpha}(y)$ such that $L_{i,\alpha}(u_i)=0$. Therefore, any $u_{i}$ satisfies a nonzero linear differential equation with coefficients in $F$, with respect to the derivation $\vartheta$. It follows from Lemma~\ref{sol commune} that the $u_{i}$ belong to $\K$. Hence, the first column of $U$ is fixed by the difference Galois group $G$ and this contradicts the fact that $G$ contains $\SL_{n}(\C)$. 
\end{proof}

\section{Applications}\label{sec applications}

In this section, we will use the notation introduced at the beginning of Section~\ref{sec hyptr}. 

\subsection{User-friendly hypertranscendence criteria}\label{user friend hyp crit}

Consider the Mahler system 
\begin{equation}\label{equa generique user-friendly syst}
\phi (Y)=AY, \text{ with } A \in \GL_{n}(\C(z)).
\end{equation}

\begin{theo}\label{user friend theo syst}
Assume that the difference Galois group of the Mahler system (\ref{equa generique user-friendly syst}) over the $\phi$-field $\K$ contains $\SL_{n}(\C)$ and that $\det A(z)$ is a monomial. Then, the following properties hold:
\begin{enumerate}
\item The parametrized difference Galois group of the Mahler system (\ref{equa generique user-friendly syst}) over $\L$ is a subgroup of $\C^{\times} \SL_{n}(\widetilde{\C})$ containing $\SL_{n}(\widetilde{\C})$.
\item Let $u=(u_{1},...,u_{n})^{t}$ be a nonzero solution of (\ref{equa generique user-friendly syst}) with entries in $\C((z))$. Then, the series $u_{1},u_{2,}\ldots,u_{n}$ and all their successive derivatives are algebraically independent over $\C(z)$. In particular, any $u_{i}$ is hypertranscendental over $\C(z)$. 
\end{enumerate}

\end{theo}

\begin{proof}
The fact that $\det A(z)$ is a monomial ensures, in virtue of Proposition \ref{propo:mahlerrang1homoge}, that the parametrized difference Galois group of $\phi(y)=\det(A)y$ is included in $\C^{\times}$.  Theorem~\ref{theo hyper alg} yields the first assertion of the theorem. 

We claim that $u_{1},u_{2,}\ldots,u_{n}$ are hyperalgebraically independent over $\C(z)$. Suppose to the contrary that they are hyperalgebraically dependent over $\C(z)$. Thanks to Lemma~\ref{lem:ppvformalsol}, there exists a PPV ring $S$ for the system (\ref{equa generique user-friendly syst}) over $\L$ containing $\K'\{u_{1},u_{2},\ldots,u_{n} \}$. Let $U \in \GL_{n}(S)$ be a fundamental matrix of solutions of the system (\ref{equa generique user-friendly syst}) whose first column is $u$. Then, $\det(U)$ is hyperalgebraic over $\L$ and the elements of the first column of $U$ are hyperalgebraically dependent over $\L$. It follows easily that the $\delta$-transcendence degree of $S$ over $\L$ is lower than or equal to $n^{2}-2$. This contradicts the fact that the $\delta$-dimension of the parametrized difference Galois group of equation (\ref{equa generique user-friendly syst}) over $\L$, namely $n^{2}-1$, is equal to the $\delta$-transcendence degree of $S$ over $\L$ (see \cite[Proposition~6.26]{HS}).
\end{proof}

We shall now state a variant of the last theorem for Mahler equations. 
Consider the following Mahler equation 
\begin{equation}\label{equa generique user-friendly}
a_{n}(z) y(z^{p^{n}}) + a_{n-1}(z) y(z^{p^{n-1}}) + \cdots + a_{0}(z) y(z) = 0
\end{equation}
for some integers $p\geq 2$, $n \geq 1$, and some $a_{0}(z),\dots,a_{n}(z) \in \C(z)$ with $a_{0}(z)a_{n}(z)\neq 0$. 
In what follows, by ``difference Galois group of equation (\ref{equa generique user-friendly})'', we mean the difference Galois group of the associated system 
\begin{equation}\label{syst generique user-friendly}
\phi (Y)=AY, \text{ with } A=\begin{pmatrix}
0&1&0&\cdots&0\\
0&0&1&\ddots&\vdots\\
\vdots&\vdots&\ddots&\ddots&0\\
0&0&\cdots&0&1\\
-\frac{a_{0}}{a_{n}}& -\frac{a_{1}}{a_{n}}&\cdots & \cdots & -\frac{a_{n-1}}{a_{n}}
\end{pmatrix} \in \GL_{n}(\C(z)).
\end{equation}

\begin{theo}\label{user friend theo}
Assume that the difference Galois group over the $\phi$-field $\K$ of the Mahler equation (\ref{equa generique user-friendly}) contains $\SL_{n}(\C)$ and that $a_{n}(z)/a_{0}(z)$ is a monomial. Then, the following properties hold:
\begin{enumerate}
\item The parametrized difference Galois group of equation (\ref{equa generique user-friendly}) over $\L$ is a subgroup of $\C^{\times} \SL_{n}(\widetilde{\C})$ containing $\SL_{n}(\widetilde{\C})$.
\item Let $f(z) \in \C((z))$ be a nonzero solution of (\ref{equa generique user-friendly}). Then, the series $f(z), f(z^{p}),\ldots,f(z^{p^{n-1}})$ and all their successive derivatives are algebraically independent over $\C(z)$. In particular, $f(z)$ is hypertranscendental over $\C(z)$. 
\end{enumerate}
 
\end{theo}

\begin{proof}
Using the fact that the determinant of the matrix $A$ given by formula (\ref{syst generique user-friendly}) is equal to $a_{0}/a_{n}$ and the fact that, if $f(z) \in \C((z))$ is a nonzero solution of (\ref{equa generique user-friendly}), then $(f(z), f(z^{p}),\ldots,f(z^{p^{n-1}}))^{t}$ is a nonzero solution of (\ref{syst generique user-friendly}) with entries in $\C((z))$, we see that this theorem is a consequence of Theorem \ref{user friend theo syst}.

\end{proof}

\subsection{The Baum-Sweet sequence} \label{section Baum Sweet}
The Baum-Sweet sequence $(a_{n})_{n\geq 0}$ is the automatic sequence defined by $a_{n} = 1$ if the binary representation of $n$ contains no block of consecutive $0$ of odd length, and $a_{n} = 0$ otherwise. It is characterized by the following recursive equations:
$$
a_{0}=1, \ \
a_{2n+1}=a_{n}, \ \
a_{4n}= a_{n}, \ \
a_{4n+2}=0.
$$
Let $f_{BS}(z)=\sum_{n\geq 0} a_{n} z^{n}$ be the corresponding generating series. The above recursive equations show that 
$$
Y(z)=\begin{pmatrix}
f_{BS}(z)  \\
f_{BS}(z^{2}) 
\end{pmatrix}
$$
satisfies the following Mahler system: 
\begin{equation}\label{syst BS}
\phi(Y)= 
A Y
\text{ where }
A = 
\begin{pmatrix}
0 & 1 \\
1 & -z
\end{pmatrix} \in \GL_{2}(\K).
\end{equation}
We have used the following notations: $p=2$, $\K= \cup_{j \geq 1} \C(z^{1/j})$ and $\phi$ is the field automorphism of $\K$ such that $\phi(z)=z^{2}$. 

\begin{theo}\label{Galois BS}
The parametrized difference Galois group of (\ref{syst BS}) over $\L$ is equal to $\mu_{4}\SL_{2}(\widetilde \C)$, where $\mu_{4} \subset \C^{\times}$ is the group of $4$th roots of the unity. The series $f_{BS}(z)$, $f_{BS}(z^{2})$ and all their successive derivatives are algebraically independent over $\C(z)$.
\end{theo}

\begin{proof}
According to \cite[Theorem 50]{Ro15}, the difference Galois group of (\ref{syst BS}) over the $\phi$-field $\K$ is equal to $\mu_{4}\SL_{2}(\C)$. Now, the result is a direct consequence of Theorem~\ref{user friend theo}. 
\end{proof}

\subsection{The Rudin-Shapiro sequence} \label{section Rudin Shapiro}

The Rudin-Shapiro sequence $(a_{n})_{n \geq 0}$ is the automatic sequence defined by $a_{n}=(-1)^{b_{n}}$ where $b_{n}$ is the number of pairs of consecutive $1$ in the binary representation of $n$. It is the characterized by the following recurrence relations: 
$$
a_{0}=1,\ \
a_{2n}=a_{n}, \ \
a_{2n+1}=(-1)^{n} a_{n}.
$$
 
Let $f_{RS}(z)=\sum_{n \geq 0} a_{n}z^{n}$ be the corresponding generating series. 
The above recursive equations show that 
$$
Y(z)=\begin{pmatrix}
f_{RS}(z)  \\
f_{RS}(-z) 
\end{pmatrix}
$$
satisfies the following Mahler system:
\begin{equation}\label{syst RS}
\phi(Y)= 
A Y
\text{ where }
A = 
\frac{1}{2}\begin{pmatrix}
1 & 1 \\
\frac{1}{z} & -\frac{1}{z}
\end{pmatrix} \in \GL_{2}(\K).
\end{equation}
We have used the following notations: $p=2$, $\K= \cup_{j \geq 1} \C(z^{1/j})$ and $\phi$ is the field automorphism of $\K$ such that $\phi(z)=z^{2}$.

\begin{theo}\label{Galois RS}
The parametrized difference Galois group of (\ref{syst RS}) over $\L$ is equal to $\GL_{2}(\widetilde \C)$. The series $f_{RS}(z)$, $f_{RS}(-z)$ and all their successive derivatives are algebraically independent over $\C(z)$.
\end{theo}

\begin{proof}
According to \cite[Theorem 54]{Ro15}, the difference Galois group of (\ref{syst RS}) over the $\phi$-field $\K$ is equal to $\GL_{2}(\C)$. Now, the result is a direct consequence of Theorem~\ref{user friend theo syst}.
\end{proof}

\subsection{Direct sum of the Baum-Sweet and of the Rudin Shapiro equations} 

The aim of this section is to illustrate how one can use the results of this paper in order to prove the hyperalgebraic independence of Mahler functions solutions of distinct equations.

\subsubsection{A differential group theoretic preliminary result}\label{subsubsec:preliminarydiffderived group}

In what follows, we let $G^{\circ}$ denote the neutral component of the linear algebraic group $G$, and we let $G^{der}$ denote its derived subgroup. We recall that $G^{\circ}$ and $G^{der}$ are Zariski-closed in $G$. 

We let $G^{\circKol}$ denote the neutral component of the differential algebraic group $G$ (so, here, we consider Kolchin's topology), and we let $G^{der}$ denote the derived subgroup of $G$. In general, $G^{der}$ is not Kolchin-closed. We let $G^{\derKol}$ denote its Kolchin-closure in $G$. 

\begin{theo} \label{GKR differential} 
Let $\k$ be a differentially closed $\delta$-field. Let $r \geq 2$ be an integer and, for any $i \in \{1,\dots,r\}$, let $G_{i}$ be an algebraic subgroup of $\GL_{n_{i}}(\k)$. We consider the linear algebraic group $G=\prod_{i \in \{1,\dots,r\}} G_{i}$. We assume that, for any $i \in \{1,\dots,r\}$, $G_{i}^{\circ,der}$ is quasi-simple and that $G^{\circ,der}=\prod_{i \in \{1,\dots,r\}} G_{i}^{\circ,der}$. Let $H$ be a Zariski-dense differential algebraic subgroup of $G$. Let $H_{i}$ be the projection of $H$ in $G_{i} \subset G$. Then,  
\begin{enumerate}
\item for all $i \in \{1,\dots,r\}$, 
$H_{i}^{\circKol,\derKol}$ is Zariski-dense in $G_{i}^{\circ,der}$; 
\item we have: 
$$H^{\circKol,\derKol}=\prod_{i \in \{1,\dots,r\}} H_{i}^{\circKol,\derKol} \subset \prod_{i \in \{1,\dots,r\}} G_{i}^{\circ,der}.$$
\end{enumerate}
\end{theo}

\begin{proof}
By hypothesis, $H$ is Zariski-dense in $G$ and, hence, $H_{i}$ is Zariski-dense in $G_{i}$ (because the projection $p_{i} : G \rightarrow G_{i}$ is continuous for the Zariski topology and, hence, $G_{i}=p_{i}(G)=p_{i}(\overline{H}) \subset \overline{p_{i}(H)}$). Therefore, $H^{\circKol,\derKol}$ is Zariski-dense in $G^{\circ,der}=\prod_{i \in \{1,\dots,r\}} G_{i}^{\circ,der}$ and $H_{i}^{\circKol,\derKol}$ is Zariski-dense in $G_{i}^{\circ,der}$. Recall that the $G_{i}^{\circ,der}$ are quasi-simple by hypothesis. It follows from \cite[Theorem 15]{CassidyCSSDAG} that 
$$
H^{\circKol,\derKol}=\prod_{i \in \{1,\dots,r\}} K_{i}
$$
for some $\delta$-closed subgroups $K_{i}$ of $ G_{i}^{\circ,der}$.
(With the terminologies of \cite[Theorem 15]{CassidyCSSDAG}, the simple components $A_{i}$ of $G^{\circ,der}$ are the $\{1\}^{i-1} \times G_{i}^{\circ,der} \times \{1\}^{r-i-1}$). We necessarily have $K_{i}=H_{i}^{\circKol,\derKol}$. 
\end{proof}
\subsubsection{Baum-Sweet and Rudin-Shapiro}

\begin{theo}\label{BS plus RS}
The parametrized difference Galois group of the direct sum of the systems (\ref{syst BS}) and (\ref{syst RS}) is equal to the direct product of the parametrized difference Galois groups of the systems (\ref{syst BS}) and (\ref{syst RS}), namely $\mu_{4}\SL_{2}(\widetilde \C) \times \GL_{2}(\widetilde \C)$. The series $f_{BS}(z),f_{BS}(z^{2}),f_{RS}(z)$, $f_{RS}(z^{2})$ and all their successive derivatives are algebraically independent over $\C(z)$.
\end{theo}

\begin{proof}
We let $M_{BS}$ and $M_{RS}$ denote the $\phi$-modules associated to the systems (\ref{syst BS}) and (\ref{syst RS}). 

It is proved in \cite[Section 9.3]{Ro15} that the difference Galois group over $\K$ of the direct sum $M_{BS} \oplus M_{RS}$ is the direct product of the difference Galois groups, {\it i.e.}, $\mu_{4} \SL(\C) \times \GL_{2}(\C)$. If follows from Theorems \ref{GKR differential}, \ref{Galois BS}, and \ref{Galois RS}, that the parametrized difference Galois group of $M_{BS} \oplus M_{RS}$ contains $\SL_{2}(\widetilde \C) \times \SL_{2}(\widetilde \C)$. The fact that the parametrized difference Galois group of  $M_{BS} \oplus M_{RS}$ is $\mu_{4}\SL_{2}(\widetilde \C) \times \GL_{2}(\widetilde \C)$ is now clear. 

The proof of the last assertion is similar to the proof of the last statement of Theorem \ref{user friend theo}. 
\end{proof}

\bibliographystyle{alpha}
\bibliography{biblio}
\end{document}